\documentclass[11pt,reqno]{amsart}
\usepackage{amsmath,amssymb,graphicx,mathrsfs,amsopn,stmaryrd,amsthm,color,bm,extarrows}

\usepackage[dvipsnames]{xcolor}
\usepackage[colorlinks=true,citecolor=blue,linkcolor=blue
]{hyperref}
\usepackage[final]{showlabels}
\usepackage[english]{babel}
\usepackage[mathscr]{euscript}
\usepackage{enumitem,etaremune}
\usepackage{geometry}
\usepackage{collectbox}
\makeatletter
\newcommand{\mybox}{%
    \collectbox{%
        \setlength{\fboxsep}{1pt}%
        \fbox{\BOXCONTENT}%
    }%
}
\makeatother

\let\emptyset\varnothing
\usepackage{dsfont}
\usepackage{tikz}    
\usetikzlibrary{arrows,matrix,decorations.markings,shapes.geometric}   
\usetikzlibrary{decorations.pathreplacing}

\usepackage{xy}
\allowdisplaybreaks
\xyoption{all}
\numberwithin{equation}{section}
\newtheorem{thm}{Theorem}[section]
\newtheorem*{mthm*}{Main Theorem}
\newtheorem{prop}[thm]{Proposition}
\newtheorem{lem}[thm]{Lemma}
\newtheorem{cor}[thm]{Corollary}

\theoremstyle{definition} 

\newtheorem{dfn}[thm]{Definition}
\theoremstyle{remark}

\newcommand{\beq}{\begin{equation}}
\newcommand{\eeq}{\end{equation}}
\newcommand{\be}{\begin{equation*}}
\newcommand{\ee}{\end{equation*}}

\newcommand{\bC}{\mathbb{C}}

\newcommand{\bZ}{\mathbb{Z}}

\newcommand{\bN}{\mathbb{N}}

\newcommand{\mc}{\mathcal}

\newcommand{\sfe}{\mathsf{e}}

\newcommand{\gl}{\mathfrak{gl}}

\newcommand{\so}{\mathfrak{so}}

\newcommand{\fko}{\mathfrak{o}}

\newcommand{\rY}{\mathrm{Y}}
\newcommand{\rU}{\mathrm{U}}

\newcommand{\End}{\mathrm{End}}
   
\newcommand{\sdet}{{\mathrm{sdet}}}   
\newcommand{\gr}{{\mathrm{gr}}}   
\newcommand{\Uq}{{\mathrm{U}_q}}  
\newcommand{\hgl}{{\mathscr{L}\mathfrak{gl}}} 
\newcommand{\hsl}{{\mathscr{L}\mathfrak{sl}}}

\newcommand{\bB}{{\mathbf B}}
\newcommand{\bTH}{{\acute{\mathbf \Theta}}}
\newcommand{\BTH}{{\mathbf \Theta}}

\newcommand{\tl}{\tilde}
\newcommand{\wtl}{\widetilde}
\newcommand{\gge}{\geqslant}
\newcommand{\lle}{\leqslant}

\newcommand{\Y}{{\mathscr{Y}}}

\newcommand{\E}{{\mathcal{E}}}
\newcommand{\F}{{\mathcal{F}}}
\newcommand{\SY}{{\mathrm{SY}_q^{\rm tw}}}

\newcommand{\bt}{{\bar{t}}}
\newcommand{\bT}{{\overline{T}}}
\newcommand{\tY}{{\mathrm{Y}_q^{\rm tw}}}
\newcommand{\Ui}{{\mathbf U^\imath}}

\newcommand{\red}[1]{{\color{red}#1}}

\begin{document}
\pagestyle{myheadings}
\setcounter{page}{1}

\title[Isomorphism between twisted $q$-Yangians and affine $\imath$quantum groups]{Isomorphism between twisted $q$-Yangians and affine $\imath$quantum groups: type AI}

\author{Kang Lu}
 \address{Department of Mathematics, University of Virginia, 
Charlottesville, VA 22903, USA}\email{kang.lu@virginia.edu}

	
\begin{abstract}
By employing Gauss decomposition, we establish a direct and explicit isomorphism between the twisted $q$-Yangians (in R-matrix presentation) and affine $\imath$quantum groups (in current presentation) associated to symmetric pair of type AI introduced by Molev-Ragoucy-Sorba and Lu-Wang, respectively. As a corollary, we obtain a PBW type basis for affine $\imath$quantum groups of type AI.
\end{abstract}
	
\maketitle

\thispagestyle{empty}
\section{Introduction}
The quantum affine algebras are well-known to have three different presentations: the Drinfeld-Jimbo presentation \cite{Dr85,Jim85}, the current (or Drinfeld's new) presentation \cite{Dr88}, and the R-matrix presentation (at least for classical types) \cite{FRT88}. Each presentation offers distinct advantages, making it crucial to establish explicit isomorphisms between them. These isomorphisms facilitate connections between various realizations and allow us to leverage the advantages from different perspectives. The isomorphism between Drinfeld-Jimbo and Drinfeld's presentations has been established in \cite{Be94, Da12, Da15}, while the isomorphism between Drinfeld's and R-matrix presentations has been demonstrated in \cite{DF93} for type A and, more recently, for types BCD in \cite{JLM20a, JLM20b}.

Recently, significant attention has been devoted to quantum symmetric pairs and the associated $\imath$quantum groups \cite{Let99, Kol14, BK20, LW21, LRW23}, with a focus on finite and Kac-Moody types. Quantum symmetric pairs $(\mathbf{U}, \mathbf{U}^\imath)$ were initially introduced in \cite{Let99} for finite types and later extended to Kac-Moody types in \cite{Kol14}. Here, $\mathbf{U}$ represents the corresponding Drinfeld-Jimbo quantum group, and $\mathbf{U}^\imath$ denotes a specific coideal subalgebra of $\mathbf{U}$ referred to as an $\imath$quantum group. The idea of generalizing all fundamental (algebraic, geometric, categorical) constructions in quantum groups to $\imath$quantum groups was proposed in \cite{BW18}. Therefore, completing an $\imath$-analog of the previous paragraph becomes both natural and essential. In this paper, we focus on the symmetric pair $(\gl_n, \so_n)$ of type AI (or split type A) which is the only case (for higher rank) where all three presentations are available in the existing literature.

The twisted $q$-Yangians of types AI and AII were introduced in \cite{MRS03} in R-matrix form whose defining relations are given by a twisted reflection equation \cite{Sk88}. These twisted $q$-Yangians represent $q$-deformations of the twisted loop algebras. The classification of finite-dimensional irreducible representations of twisted $q$-Yangians for type AII can be found in \cite{GM10}, while the classification problems for types AI and AIII (introduced in \cite{CGM14}) are still open. In \cite{Kol14}, the affine $\imath$quantum groups of types AI and AII (Serre-type presentation) were identified with the twisted $q$-Yangians in R-matrix presentation. In \cite{LW21}, a Drinfeld type presentation of the affine $\imath$quantum group of type AI was introduced and an isomorphism between the Drinfeld-type and the Serre-type presentations are demonstrated. Explicit isomorphisms are provided in both cases. The objective of this paper is to complete this picture by providing an explicit isomorphism between the current and R-matrix presentations.

Our approach relies on the well-known Gauss decomposition, which has been successfully applied to establish isomorphisms between R-matrix and Drinfeld's (current) presentations for quantum affine algebras and Yangians of classical types \cite{DF93, BK05, JLM18, JLM20a, JLM20b}. As confirmed in \cite{CG15}, the twisted $q$-Yangian degenerates to twisted Yangians \cite{Ol92} in their R-matrix presentations. Therefore, the recent joint work with Weiqiang Wang and Weinan Zhang on Gauss decomposition of twisted Yangians (of type AI) in R-matrix presentation and degeneration of affine $\imath$quantum groups in Drinfeld's presentation \cite{LWZ23a, LWZ23b,LZ24} has been instrumental in guiding this research.

While we expect the Gauss decomposition to be applicable to twisted $q$-Yangians of other classical types, it is undeniably more challenging. It should be noted that R-matrix presentations for affine $\imath$quantum groups are currently available only for types AI, AII, and AIII \cite{MRS03, CGM14}. An interesting future direction would involve investigating the R-matrix construction for classical types \cite{RV16}, as explored in \cite{GR16} for twisted Yangians.

Our results are anticipated to provide applications to the study of open spin chains \cite{Sk88} of XXZ type. For instance, they hold the potential to facilitate the computation of weight functions (off-shell Bethe vectors) for quantum integrable models with boundary conditions (associated to twisted $q$-Yangians) via the method of projection of Drinfeld currents as developed in \cite{EKP07}. Consequently, our results pave the way for the advancement of the zero modes method, as introduced by \cite{LPRS19}, in the investigation of open spin chains. This advancement encompasses tasks such as the calculation of scalar products of states and the determination of form-factors for local operators.

\medskip

The organization of this paper is as follows: Section \ref{sec:yangian} provides a recap of the definition and fundamental properties of quantum loop algebras and twisted $q$-Yangians. In Section \ref{sec:Gauss}, we apply the Gauss decomposition to introduce new generators of twisted $q$-Yangians and investigate their basic properties. Section \ref{sec:main} is dedicated to presenting the main theorem and its applications, including a PBW theorem. Finally, in Section \ref{sec:low}, we verify the relations of twisted $q$-Yangians of ranks 1 and 2. 

\medskip 

{\bf Acknowledgement}. The author thanks Weiqiang Wang and Weinan Zhang for collaborations and stimulating discussions. The author is partially supported by NSF grants DMS-2001351 and DMS--2401351, both awarded to Weiqiang Wang.




\section{Twisted $q$-Yangians of type AI}\label{sec:yangian}
In this section, we recall the basics about quantum loop algebras and twisted $q$-Yangians, following \cite{MRS03}. Throughout the paper, we assume that $q$ is a formal variable and for each $k\in\bZ_{\gge 0}$, we use the standard notation
\[
[k]=\frac{q^k-q^{-k}}{q-q^{-1}}.
\]
For $x,y$ in a $\bC(q)$-algebra, we shall denote $[x,y]_{q^a}=xy-q^ayx$, and $[x,y]=xy-yx$.

\subsection{Trigonometric R-matrices}
Let $n$ be a fixed positive integer. Define the trigonometric R-matrix depending on two parameters,
\beq\label{Ru}
\begin{split}
    R(u,v)= (u-v)\sum_{i\ne j}E_{ii}\otimes E_{jj}&+(q^{-1}u-qv)\sum_{1\lle i\lle n}E_{ii}\otimes E_{ii}\\
    &+(q^{-1}-q)u\sum_{i>j}E_{ij}\otimes E_{ji}+(q^{-1}-q)v\sum_{i<j}E_{ij}\otimes E_{ji},
\end{split}
\eeq
where $E_{ij}\in \End(\bC^n)$ are the standard matrix units. It would be convenient to work with one-parameter R-matrices,
\beq\label{bRx}
\begin{split}
    \overline R(x)= \frac{R(x,1)}{q^{-1}x-1}=&\, \sum_{1\lle i\lle n}E_{ii}\otimes E_{ii}+\frac{1-x}{q-q^{-1}x}\sum_{i\ne j}E_{ii}\otimes E_{jj}\\
    &+\frac{(q^{-1}-q)x}{q-q^{-1}x}\sum_{i>j}E_{ij}\otimes E_{ji}+\frac{q^{-1}-q}{q-q^{-1}x}\sum_{i<j}E_{ij}\otimes E_{ji},
\end{split}
\eeq
and
\beq\label{Rx}
R(x)=f(x)\overline R(x),
\eeq
where
\[
f(x)=1+\sum_{k=1}^\infty f_kx^k,\qquad f_k=f_k(q),
\]
is a formal power series in $x$ whose coefficients $f_k$ are rational functions in $q$ and $f(x)$ is uniquely determined by the relation
\[
f(xq^{2n})=f(x)\frac{(1-xq^2)(1-xq^{2n-2})}{(1-x)(1-xq^{2n})}.
\]

Let $R^t(u,v):=R^{t_1}(u,v)$ be obtained from $R(u,v)$ by taking the transpose in the first factors,
\beq\label{Rt}
\begin{split}
    R^t(u,v)= (u-v)\sum_{i\ne j}E_{ii}\otimes E_{jj}&+(q^{-1}u-qv)\sum_{1\lle i\lle n}E_{ii}\otimes E_{ii}\\
    &+(q^{-1}-q)u\sum_{i>j}E_{ji}\otimes E_{ji}+(q^{-1}-q)v\sum_{i<j}E_{ji}\otimes E_{ji},
\end{split}
\eeq
and define similar notations for other R-matrices. Denote $D$ the diagonal $n\times n$ matrix,
\beq\label{D}
D=\mathrm{diag}(q^{n-1},q^{n-2},\cdots,q,1).
\eeq

We collect some basic properties of the trigonometric R-matrix. The R-matrix satisfies the \textit{Yang-Baxter equation},
\beq\label{ybeq}
R_{12}(u,v)R_{13}(u,w)R_{23}(v,w)=R_{23}(v,w)R_{13}(u,w)R_{12}(u,v),
\eeq
the \textit{unitary property},
\beq\label{unitary}
\overline R(x^{-1})=\overline R_{21}(x)^{-1},
\eeq
and the \textit{crossing symmetry relation} \cite{FR92},
\beq\label{cross}
R^{t_1}(xq^{2n})D_1D_2 \big(R(x)^{-1}\big)^{t_1}=D_1D_2.
\eeq
Moreover, $R(u,v)$ and $R^t(u^{-1},v)$ commute,
\beq\label{RRtcom}
R(u,v)R^t(u^{-1},v)=R^t(u^{-1},v)R(u,v).
\eeq

\subsection{Quantum loop algebras}
\label{subsec:loop}
\begin{dfn}
The \textit{quantum loop algebra} $\Uq(\hgl_n)$ corresponding to the Lie algebra $\gl_n$ is a unital associative $\bC(q)$-algebra with generators $t_{ij}^{(r)}$ and $\bar t_{ij}^{(r)}$, where $1\lle i,j\lle n$ and $r\in\bZ_{\gge 0}$, and the defining relations given as follows. Set the generating series
\[
t_{ij}(u)=  \sum_{r\gge 0}t_{ij}^{(r)}u^{-r},\qquad \bar t_{ij}(u)=\sum_{r\gge 0}\bar t_{ij}^{(r)}u^{r}
\]
and combine them into matrices
\[
T(u)=\sum_{i,j=1}^n t_{ij}(u)\otimes E_{ij},\qquad \overline{T}(u)=\sum_{i,j=1}^n \bar t_{ij}(u)\otimes E_{ij}.
\]
Then the defining relations of $\Uq(\hgl_n)$ are
\begin{align}
    t_{ij}^{(0)}&=\bt_{ji}^{(0)}=0,\qquad 1\lle i<j\lle n,\nonumber\\
    t_{ii}^{(0)}\bt_{ii}^{(0)}&=\bt_{ii}^{(0)}t_{ii}^{(0)}=1,\qquad 1\lle i\lle n,\label{zerocartan}\\
    R(u,v)T_1(u)T_2(v)&=T_2(v)T_1(u)R(u,v),\nonumber\\
    R(u,v)\overline T_1(u)\overline T_2(v)&=\overline T_2(v)\overline T_1(u)R(u,v),\nonumber\\
    R(u,v)\overline T_1(u)T_2(v)&=T_2(v)\overline T_1(u)R(u,v)\nonumber.
\end{align}
Here both sides of RTT relations are series with coefficients in $\Uq(\hgl_n)\otimes \End(\bC^n)\otimes \End(\bC^n)$ and the subscripts of $T(u)$ and $\overline T(u)$ indicate the copies of $\End(\bC^n)$ where $T(u)$ or $\overline T(u)$ acts; e.g. $T_1(u)=T(u)\otimes 1$.
\end{dfn}

Let $\gl_n[z^{\pm 1}]$ be the loop algebra $\gl_n\otimes\bC[z^{\pm 1}]$ in an indeterminate $z$. The quantum loop algebra $\Uq(\hgl_n)$ is a deformation of the universal enveloping algebra $\rU(\gl_n[z^{\pm 1}])$ in the sense that $\Uq(\hgl_n)$ specializes to $\rU(\gl_n[z^{\pm 1}])$ as $q\to 1$. Specifically, regard $q$ as a formal variable and $\rU(\gl_n[z^{\pm 1}])$ as an algebra over $\bC(q)$. Set $\mc A=\bC[q^{\pm 1}]$. Let $\rU_{\mc A}$ be the $\mc A$-subalgebra of $\Uq(\hgl_n)$ generated by the elements $\tau_{ij}^{(r)}$ and $\bar\tau_{ij}^{(r)}$ defined by
\beq\label{tau}
\tau_{ij}^{(r)}=\frac{t_{ij}^{(r)}}{q-q^{-1}},\qquad \bar\tau_{ij}^{(r)}=\frac{\bt_{ij}^{(r)}}{q-q^{-1}},
\eeq
for $r\gge 0$ and all $1\lle i,j\lle n$, except for the case $r=0$ and $i=j$ where we set
\beq\label{tau0}
\tau_{ii}^{(0)}=\frac{t_{ii}^{(0)}-1}{q-1},\qquad \bar \tau_{ii}^{(0)}=\frac{\bt_{ii}^{(0)}-1}{q-1}.
\eeq
Then there exists an isomorphism
\beq\label{classical}
\rU_{\mc A}\otimes_{\mc A}\bC\cong \rU(\gl_n[z^{\pm 1}]),
\eeq
with the action of $\mc A$ on $\bC$ defined via the evaluation $q = 1$, see e.g. \cite[\S 12.2]{CP95} and \cite[\S 2]{FM02}. Moreover, the images of the generators of $\rU_{\mc A}$ under the isomorphism \eqref{classical} are given by
\beq\label{imglimit}
\tau_{ij}^{(r)}\to \sfe_{ij}\otimes z^r,\qquad \bar \tau_{ij}^{(r)}\to -\sfe_{ij}\otimes z^{-r}
\eeq
for all $r\gge 0$ with the exception $\tau_{ij}^{(0)}=\bar\tau_{ji}^{(0)}=0$ if $1\lle i<j\lle n$. Here $\sfe_{ij}$, $1\lle i,j\lle n$, are the standard elements of the Lie algebra $\gl_n$.

For a subalgebra $V$ of $\Uq(\hgl_n)$, we set $V_{\mc A}=V\cap \rU_{\mc A}$. We say that $V$ specializes to a subalgebra $V^{\circ}$ of $\rU(\gl_n[z^{\pm1}])$ if $V_{\mc A}\otimes_{\mc A}\bC\cong V^\circ$.

We shall need the negative half $\rU_q^{<}(\hsl_n)$ of $\rU(\hsl_n)$ in terms of Drinfeld's current presentation. The $\bC(q)$-algebra $\rU_q^{<}(\hsl_n)$ is generated by $\xi_{i,r}$ where $1\lle i<n$ and $r\in\bZ$ subject to the following relations:
\begin{align*}
&[\xi_{i,k},\xi_{j,l}]=0,\hskip5.7cm \text{ if }c_{ij}=0,\\
&[\xi_{i,k},\xi_{j,l+1}]_{q^{-c_{ij}}}-q^{-c_{ij}}[\xi_{i,k+1},\xi_{j,l}]_{q^{c_{ij}}}=0,\\
&\xi_{i,k_1}\xi_{i,k_2}\xi_{j,l}-[2]\xi_{i,k_1}\xi_{j,l}\xi_{i,k_2}+\xi_{j,l}\xi_{i,k_1}\xi_{i,k_2}+\{k_1\leftrightarrow k_2\}=0,
\end{align*}
see e.g. \cite{Lu93} and \cite[Theorem 2]{Her05}. Here and below, $\{k_1\leftrightarrow k_2\}$ stands for repeating the previous summands with $k_1,k_2$ interchanged, so the sum over $k_1,k_2$ are symmetrized.

For any $1\lle i<j\lle n$, define elements $\xi_{ji,r}$ in $\rU_q^{<}(\hsl_n)$ inductively by the rule:
\beq\label{bijr-A}
\xi_{i+1\, i,r}=\xi_{i,r},\qquad \xi_{ji,r}=[\xi_{j-1,0},\xi_{j-1\, i,r}]_q,
\eeq
if $j-i\gge 2$. 

Define a total order on the elements $\xi_{ji,r}$ by the rule:
\beq\label{total}
\xi_{ji,r}\prec \xi_{j'i',r'} \text{ iff } (i'<i) \text{ or } (i'=i, j'< j) \text{ or } (i'=i, j'=j, r'<r).
\eeq
The order is opposite to the one in \cite[\S2.2]{Tsy21} due to \cite[(2.15)]{Tsy21}.
\begin{prop}[{\cite[Theorem 2.16]{Tsy21}}]\label{PBW-A}
The ordered monomials in the elements $\xi_{ji,r}$ form a $\bC(q)$-basis of $\rU_q^{<}(\hsl_n)$.
\end{prop}

\subsection{Twisted $q$-Yangians}
In this subsection, we recall the definition of twisted $q$-Yangians from \cite[\S 3]{MRS03}.

\begin{dfn}\label{tydef}
The \textit{twisted $q$-Yangian} $\tY(\fko_n)$ is the $\bC(q)$-subalgebra of $\Uq(\hgl_n)$ generated by the coefficients $s_{ij}^{(r)}$ of the entries of the matrix $S(u)=T(u)\bT(u^{-1})^t$. Namely, let
\beq\label{quater u}
s_{ij}(u)=\sum_{r\gge 0}s_{ij}^{(r)}u^{-r}:=\sum_{a=1}^n t_{ia}(u)\bt_{ja}(u^{-1}).
\eeq
The subalgebra $\tY(\fko_n)$ is generated by the elements $s_{ij}^{(r)}$ for $1\lle i,j\lle n$ and $r\in \bZ_{\gge 0}$.
\end{dfn}

The twisted $q$-Yangian $\tY(\fko_n)$ also admits an R-matrix presentation as follows. 

\begin{dfn}
The \textit{twisted $q$-Yangian} $\Y_q^{\rm tw}(\fko_n)$ is a unital associative $\bC(q)$-algebra generated by $s_{ij}^{(r)}$ with $1\lle i,j\lle n$ and $r\in\bZ_{\gge 0}$ and the defining relations are given by
\begin{align}
s_{ij}^{(0)}&=0,\qquad 1\lle i<j\lle n,\label{ij0}\\
s_{ii}^{(0)}&=1,\qquad 1\lle i\lle n,\label{ii0}\\
R(u,v)S_1(u)R^{t}(u^{-1},v)S_2(v)&=S_2(v)R^{t}(u^{-1},v)S_1(u)R(u,v).\label{quat-mat}
\end{align}
where $R^{t}(u,v)$ is defined in \eqref{Rt} and
\[
s_{ij}(u)=\sum_{r\gge 0}s_{ij}^{(r)}u^{-r},\qquad S(u)=\sum_{i,j=1}^n s_{ij}(u)\otimes E_{ij}.
\]
\end{dfn}

It is known from \cite[Theorem 3.3]{MRS03} that the map
\[
\tY(\fko_n)\to \Y_q^{\rm tw}(\fko_n),\quad s_{ij}^{(r)}\mapsto s_{ij}^{(r)}
\]
defines an isomorphism of algebras. Thus we shall write $\tY(\fko_n)$ for $\Y_q^{\rm tw}(\fko_n)$.

In terms of generating series, the defining relations can be written as,
\beq\label{siju}
\begin{split}
(q^{-\delta_{ij}}u&-q^{\delta_{ij}}v)\alpha_{ijab}(u,v)+(q^{-1}-q)(u\delta_{j<i}+v\delta_{i<j})\alpha_{jiab}(u,v)\\
&=(q^{-\delta_{ab}}u-q^{\delta_{ab}}v)\alpha_{jiba}(v,u)+(q^{-1}-q)(u\delta_{a<b}+v\delta_{b<a})\alpha_{jiab}(v,u),
\end{split}
\eeq
where the notation $\alpha_{ijab}(u)$ is defined by
\beq\label{alphau}
\begin{split}
\alpha_{ijab}(u)=(q^{-\delta_{aj}}-q^{\delta_{aj}}uv)s_{ia}(u)s_{jb}(v)+(q^{-1}-q)(\delta_{j<a}+uv\delta_{a<j})s_{ij}(u)s_{ab}(v).
\end{split}
\eeq

If $1\lle i<j<n$, then by \eqref{siju} we have
\beq\label{eq:zero-mode-s}
\begin{split}
[s_{ij}(u),s_{j+1,j}^{(0)}]_q=\big(q^{-1}-q\big)s_{i,j+1}(u),\\
[s_{ji}(u),s_{j+1,j}^{(0)}]_q=\big(q^{-1}-q\big)s_{j+1,i}(u).   
\end{split}
\eeq

\subsection{Some basic properties}
Let $\theta$ be the involution of $\gl_n$ defined by
\[
\theta: \gl_n\longrightarrow \gl_n,\quad e_{ij}\mapsto -e_{ji}.
\]
Extend this involution to an involution on the loop algebra $\gl_n[z^{\pm1}]$ (denoted again by $\theta$) sending 
\begin{align*}
    \mathtt g\otimes z^r \mapsto \theta(\mathtt g)\otimes z^{-r} \qquad \qquad (\mathtt g\in \gl_n, \; r\in\bZ). 
\end{align*}
Denote by $\gl_n[z^{\pm 1}]^\theta$ the fixed point subalgebra of $\gl_n[z^{\pm 1}]$ under the involution $\theta$. Set $V=\tY(\fko_n)$. Then the algebra $V_{\mc A}$ is generated by the elements 
\beq\label{sigmaij}
\sigma_{ij}^{(r)}=\frac{s_{ij}^{(r)}}{q-q^{-1}},
\eeq
where we assume further that $i>j$ if $r=0$, cf. \eqref{ij0} and \eqref{ii0}. Then the subalgebra $\tY(\fko_n)$ specializes to $\rU(\gl_n[z^{\pm 1}]^\theta)$, i.e.
\beq\label{spec}
V_{\mc A}\otimes_{\mc A}\bC\cong \rU(\gl_n[z^{\pm 1}]^\theta).
\eeq
Moreover, under this isomorphism, the image of $\sigma_{ij}^{(r)}$ in $V_{\mc A}\otimes_{\mc A}\bC$ is identified with
\beq\label{sigmaimg}
\sigma_{ij}^{(r)}\mapsto \sfe_{ij}\otimes z^r-\sfe_{ji}\otimes z^{-r}.
\eeq

In particular, we have the well-known PBW type theorem. Define an ordering on the generators such that $s_{ij}^{(r)}\preceq s_{kl}^{(p)}$ if and only if $(i,j,r)\preceq(k,l,p)$ in the lexicographical order.
\begin{prop}[{\cite[Corollary 3.4]{MRS03}}]\label{prop:PBW}
The ordered monomials in the elements
\begin{align*}
&s_{ij}^{(0)},\qquad 1\lle j<i\lle n,\\
&s_{ij}^{(r)},\qquad 1\lle i,j\lle n,\, r\in\bZ_{>0}
\end{align*}
form a basis of the twisted $q$-Yangians $\tY(\fko_n)$.
\end{prop}

Let $\wtl S(u)=S(u)^{-1}=(\tl s_{ij}(u))_{1\lle i,j\lle n}$. 

\begin{lem}\label{stslem}
If $\{i,j\}\cap \{k,l\}=\emptyset$, then $[s_{ij}(u),\tl s_{kl}(v)]=0$.
\end{lem}
\begin{proof}
By \eqref{quat-mat}, we have
\[
\wtl S_2(v) R(u,v)S_1(u)R^t(u^{-1},v)=R^t(u^{-1},v)S_1(u)R(u,v)\wtl S_2(v).
\]
The lemma follows from the above equation by comparing the coefficients of $E_{ij}\otimes E_{kl}$.
\end{proof}

By the defining relations in terms of generating series $s_{ij}(u)$ from \eqref{siju} and \eqref{alphau}, given any $m\in \bN$, we have the natural homomorphism $\jmath$ from $\tY(\fko_{n})$ to $\tY(\fko_{m+n})$ given by
\begin{align}  \label{jmath}
\jmath: \tY(\fko_{n}) \longrightarrow \tY(\fko_{m+n}),\quad s_{ij}(u)\mapsto s_{ij}(u).
\end{align}
There is also a homomorphism $\iota_m$ from $\tY(\fko_{n})$ to $\tY(\fko_{m+n})$ given by shifting indices,
\begin{align} \label{iota}
\iota_m:\tY(\fko_{n})\longrightarrow \tY(\fko_{m+n}),\quad s_{ij}(u)\mapsto s_{m+i,m+j}(u).
\end{align}

Note that the notations $R(x)$, $S(u)$ etc. depend on the parameter $q$. One can also define these notations corresponding to the parameter $q^{-1}$. To distinguish the notations for $q$ and for $q^{-1}$. We shall write $\mathscr R(x)$, $\mathscr S(u)$ etc. for $q^{-1}$. In particular, $D\mathscr D=\mathbb{I}_n$ is the identity matrix.

\begin{prop}
The map
\beq\label{varpi}
\varpi_{q,n}: \tY(\fko_{n})\to \rY_{q^{-1}}^{\rm tw}(\fko_n), \qquad S(u)\to D\wtl{\mathscr S}(uq^n)D^{-1}
\eeq
defines an algebra isomorphism, where $\wtl{\mathscr S}(u)=\mathscr S(u)^{-1}$. Moreover, $\varpi_{q^{-1},n}\circ\varpi_{q,n}=\mathrm{Id}$.
\end{prop}
\begin{proof}
By definition, we have
\beq\label{pfauo}
\mathscr R(u,v)\mathscr S_1(uq^n)\mathscr R^{t_1}(u^{-1}q^{-n},vq^n)\mathscr S_2(vq^n)=\mathscr S_2(vq^n)\mathscr R^{t_1}(u^{-1}q^{-n},vq^n)\mathscr S_1(uq^n)\mathscr R(u,v),
\eeq
where $\mathscr R(u,v)$ (resp. $\mathscr R(x)$) is equal to $R(u,v)$ (resp. $R(x)$)with $q$ replaced by $q^{-1}$. Note that the matrices $\mathscr R(u,v)$ and $R(u,v)^{-1}$ are proportional (up to a rational function in $u,v,q$). Also, it follows from \eqref{cross} that
\begin{align*}
\big(\mathscr R^{t_1}(u^{-1}v^{-1}q^{-2n})\big)^{-1}&\propto \Big(\big(R(u^{-1}v^{-1}q^{-2n})^{-1}\big)^{t_1}\Big)^{-1}\\
&=D_1^{-1}D_2^{-1}R^{t_1}(u^{-1}v^{-1})D_1D_2,
\end{align*}
where $\propto$ stands for ``proportional to up to a scalar series".
Thus we have
\[
\big(\mathscr R^{t_1}(u^{-1}q^{-n},vq^n)\big)^{-1}\propto D_1^{-1}D_2^{-1}R^{t_1}(u^{-1},v)D_1D_2.
\]
Taking the inverse of equation \eqref{pfauo} and using the above observations, we have
\begin{align*}
\wtl{\mathscr S}_2(vq^n)D_1^{-1}D_2^{-1}&R^{t_1}(u^{-1},v)D_1D_2\wtl{\mathscr S}_1(uq^n)R(u,v)\\
=\,&R(u,v)\wtl{\mathscr S}_1(uq^n)D_1^{-1}D_2^{-1}R^{t_1}(u^{-1},v)D_1D_2\wtl{\mathscr S}_2(vq^n).	
\end{align*}
Multiply $D_1D_2$ from left and $D_1^{-1}D_2^{-1}$ from right. Using the equality
\[
D_1D_2R(u,v)=R(u,v)D_1D_2,\qquad D_1^{-1}D_2^{-1}R(u,v)=R(u,v)D_1^{-1}D_2^{-1},
\]
we conclude that
\begin{align*}
D_2\wtl{\mathscr S}_2(vq^n)D_2^{-1}&R^{t_1}(u^{-1},v)D_1\wtl{\mathscr S}_1(uq^n)D_1^{-1}R(u,v)\\
=\,&R(u,v)D_1\wtl{\mathscr S}_1(uq^n)D_1^{-1}R^{t_1}(u^{-1},v)D_2\wtl{\mathscr S}_2(vq^n)D_2^{-1},
\end{align*}
which implies the map $\varpi_{q,n}$ is an algebra homomorphism. Note that
\[
\varpi_{q,n}:\wtl S(u)\to D \mathscr S(uq^n)D^{-1},
\]
we find that
\begin{align*}
&\varpi_{q^{-1},n}\circ\varpi_{q,n}(\wtl S(u))= \varpi_{q^{-1},n}(D \mathscr S(uq^n)D^{-1})\\=\,&D\varpi_{q^{-1},n}(\mathscr S(uq^n))D^{-1}
=D\mathscr D \wtl S((uq^{n})q^{-n})\mathscr D^{-1}D^{-1}=\wtl S(u).  
\end{align*}
Thus $\varpi_{q^{-1},n}\circ\varpi_{q,n}=\mathrm{Id}$. Similarly, $\varpi_{q,n}\circ\varpi_{q^{-1},n}=\mathrm{Id}$. Therefore, $\varpi_{q,n}$ is an isomorphism.
\end{proof}

\section{Gauss decomposition}\label{sec:Gauss}

In this section, we introduce Gauss decomposition for the twisted $q$-Yangian $\tY(\fko_N)$ and formulate several basic properties of the new current generators.

\subsection{Quasi-determinants and Gauss decomposition}
Let $X$ be a square matrix over a ring with identity such that its inverse matrix $X^{-1}$ exists, and such that its $(j,i)$-th entry is an invertible element of the ring.  Then the $(i,j)$-th \emph{quasi-determinant} of $X$ is defined by the first formula below and denoted graphically by the boxed notation (cf. \cite[\S1.10]{Mol07}):
\begin{equation*}
\vert X\vert _{ij} \stackrel{\text{def}}{=} \left((X^{-1})_{ji}\right)^{-1} = \left\vert  \begin{array}{ccccc} x_{11} & \cdots & x_{1j} & \cdots & x_{1n}\\
&\cdots & & \cdots&\\
x_{i1} &\cdots &\boxed{x_{ij}} & \cdots & x_{in}\\
& \cdots& &\cdots & \\
x_{n1} & \cdots & x_{nj}& \cdots & x_{nn}
\end{array} \right\vert .
\end{equation*}

By \cite[Theorem 4.96]{GGRW:2005}, the matrix $S(u)$ has the following Gauss decomposition:
$$
S(u) = F(u) D(u) E(u)
$$
for unique matrices of the form
\[
F(u)=\begin{bmatrix}
1&0&\dots&0~\\
f_{21}(u)&1&\dots&0~\\
\vdots&\vdots&\ddots&\vdots~\\
f_{n1}(u)&f_{n2}(u)&\dots&1~
\end{bmatrix},
\qquad
E(u)=\begin{bmatrix}
~1&e_{12}(u)&\dots&e_{1n}(u)\\
~0&1&\dots&e_{2n}(u)\\
~\vdots&\vdots&\ddots&\vdots\\
~0&0&\dots&1
\end{bmatrix},
\]
and $D(u)=\mathrm{diag}\big[d_1(u),\dots,d_{n}(u)\big]$, where the matrix entries are defined in terms of quasi-determinants:
\begin{eqnarray*}
d_i (u) &=& \left|\begin{matrix} s_{11}(u) &\cdots &s_{1,i-1}(u) &s_{1i}(u) \\
\vdots &\ddots & &\vdots \\
s_{i1}(u) &\cdots &s_{i,i-1}(u) &\boxed{s_{ii}(u)}
\end{matrix}\right| , 
\qquad \tl d_i(u)=d_i(u)^{-1}, 
\\
e_{ij}(u) &=&\tl d_i (u) \left|\begin{matrix} s_{11}(u) &\cdots & s_{1,i-1}(u) & s_{1j}(u) \\
\vdots &\ddots &\vdots & \vdots \\
s_{i-1,1}(u) &\cdots &s_{i-1,i-1}(u) & s_{i-1,j}(u)\\
s_{i1}(u) &\cdots &s_{i,i-1}(u) &\boxed{s_{ij}(u)}
\end{matrix} \right\vert,
\\
f_{ji}(u) &=& \left|\begin{matrix} s_{11}(u) &\cdots &s_{1, i-1}(u) & s_{1i}(u) \\
\vdots &\ddots &\vdots &\vdots \\
s_{i-1,1}(u) &\cdots &s_{i-1,i-1}(u) &s_{i-1,i}(u)\\
s_{j1}(u) &\cdots &s_{j, i-1}(u) &\boxed{s_{ji}(u)} 
\end{matrix} \right\vert\, 
\tl d_{i}(u).
\end{eqnarray*}
The Gauss decomposition can also be written component-wise as, for $i<j$, 
\begin{align}
s_{ii}(u)&=d_i(u)+\sum_{k<i}f_{ik}(u)d_k(u)e_{ki}(u),\nonumber\\
s_{ij}(u)&=d_i(u)e_{ij}(u)+\sum_{k<i}f_{ik}(u)d_k(u)e_{kj}(u),\label{eq:sij-Gauss}\\
s_{ji}(u)&=f_{ji}(u)d_i(u)+\sum_{k<i}f_{jk}(u)d_k(u)e_{ki}(u).\nonumber
\end{align}

We further denote
\begin{align}
    \label{gauss-gen}
e_{ij}(u) &=\sum_{r\gge 1}e_{ij}^{(r)}u^{-r},\qquad\hskip0.3cm f_{ji}(u)=\sum_{r\gge 0}f_{ji}^{(r)}u^{-r},\qquad\hskip0.3cm d_k(u)=1+\sum_{r\gge 1}d_{k}^{(r)}u^{-r},
\\
e_i(u) &=\sum_{r\gge 1}e_{i}^{(r)}u^{-r}=e_{i,i+1}(u),\quad f_i(u)=\sum_{r\gge 0}f_{i}^{(r)}u^{-r}=f_{i+1,i}(u),\quad 1\lle i<n.
\end{align}

Set 
\beq\label{edfinv}
\begin{split}
&\wtl D(u)=D(u)^{-1}=\sum_{1\lle i\lle n} E_{ii}\otimes \tl d_{i}(u),
\\
\wtl E(u)=E(u)^{-1}=\sum_{1\lle i<j\lle n}&E_{ij}\otimes \tl e_{ij}(u),\qquad \wtl F(u)=F(u)^{-1}=\sum_{1\lle i<j\lle n}E_{ji}\otimes \tl f_{ji}(u).
\end{split}
\eeq
Then we have 
\beq\label{eq:def-tilde-e-f}
\begin{split}
&\tl e_{ij}(u)=\sum_{i=i_0<i_1<\cdots<i_s=j}(-1)^s e_{i_0i_1}(u)e_{i_1i_2}(u)\cdots e_{i_{s-1}i_s}(u),\\
&\tl f_{ji}(u)=\sum_{i=i_0<i_1<\cdots<i_s=j}(-1)^s f_{i_{s}i_{s-1}}(u)\cdots f_{i_2i_1}(u) f_{i_1i_0}(u).
\end{split}
\eeq
Multiplying out the product $T(u)^{-1}=E(u)^{-1}D(u)^{-1}F(u)^{-1}$, we have
\beq\label{sij-Gauss-Inv}
\begin{split}
\tl s_{ii}(u)&=\tl d_i(u)+\sum_{k>i}\tl e_{ik}(u)\tl d_k(u)\tl f_{ki}(u),\\
\tl s_{ij}(u)&=\tl e_{ij}(u)\tl d_j(u)+\sum_{k>j}\tl e_{ik}(u)\tl d_k(u)\tl f_{kj}(u),\\
\tl s_{ji}(u)&=\tl d_j(u)\tl f_{ji}(u)+\sum_{k>j}\tl e_{jk}(u)\tl d_k(u)\tl f_{ki}(u).
\end{split}
\eeq

\subsection{Properties of new current generators}
Recall $\iota$ and $\varpi_{q,n}$ from \eqref{iota} and \eqref{varpi}. Define a homomorphism $\vartheta_m:\tY(\fko_n) \to \tY(\fko_{m+n})$ 
as the composition
\beq\label{eq:theta-m}
\vartheta_m= \varpi_{q^{-1},m+n}\circ \iota_m\circ \varpi_{q,n} \colon \tY(\fko_n)\longrightarrow \tY(\fko_{m+n}).
\eeq
\begin{lem}\label{lem:theta}
For any $1\lle i,j\lle n$, $\vartheta_m$ maps $s_{ij}(uq^m)$ to
\[
\left\vert  \begin{array}{cccc} s_{11}(u) &\cdots &s_{1m}(u) &s_{1, m+j}(u)\\
\vdots &\ddots &\vdots &\vdots \\
s_{m1}(u) &\cdots &s_{mm}(u) &s_{m, m+j}(u)\\
s_{m+i, 1}(u) &\cdots &s_{m+i,m}(u) &\mybox{$s_{m+i,m+j}(u)$}
\end{array} \right\vert.
\]
\end{lem}
\begin{proof}
By tracing the definition $\vartheta_m$ in \eqref{eq:theta-m} through \eqref{iota} and \eqref{varpi}, one sees that $\vartheta_m$ maps $\tl s_{ij}(u)$ to $\tl s_{m+i,m+j}(uq^{-m})$. The rest of the proof is similar to that of \cite[Lemma 2.14.1]{Mol07}.
\end{proof}

\begin{cor}\label{theta}
The map $\vartheta_m\colon \tY(\fko_n)\longrightarrow \tY(\fko_{m+n})$ sends 
\[
d_i(uq^m)\mapsto d_{m+i}(u), \quad
e_{ij}(uq^m)\mapsto e_{m+i,m+j}(u), \quad
f_{ji}(uq^m)\mapsto f_{m+j,m+i}(u).
\]
\end{cor}

\begin{proof}
Since $\vartheta_m$ maps $\tl s_{ij}(u)$ to $\tl s_{m+i,m+j}(uq^{-m})$, it follows that 
\begin{align} \label{l12}
    \vartheta_{\ell_1}\circ \vartheta_{\ell_2}=\vartheta_{\ell_1+\ell_2}, \qquad (\ell_1,\ell_2\gge 0); 
\end{align}
here and below it is understood that $n$ is fixed throughout when considering $\vartheta_m$ for various $m$. By Lemma \ref{lem:theta}, we have
\begin{align*}
d_i(u)&=\vartheta_{i-1}(s_{11}(uq^{i-1}))\\
e_{ij}(u)&=\vartheta_{i-1}(s_{11}\big(uq^{i-1})^{-1}s_{1,j-i+1}(uq^{i-1})\big),\\
f_{ji}(u)&=\vartheta_{i-1}\big(s_{j-i+1,1}(uq^{i-1})s_{11}(uq^{i-1})^{-1}\big).
\end{align*}
Then the lemma follows from applying \eqref{l12}.
\end{proof}

Recall the homomorphism $\jmath: \tY(\fko_m) \rightarrow \tY(\fko_{m+n})$ from \eqref{jmath}.

\begin{lem}\label{comsubalg}
The subalgebras $\jmath(\tY(\fko_m))$ and $\vartheta_m(\tY(\fko_{n}))$ of $\tY(\fko_{m+n})$ commute with each other.
\end{lem}
\begin{proof}
Note that the map $\vartheta_m$ sends $\tl s_{ij}(u)$ in $\tY(\fko_n)$ to $\tl s_{m+i,m+j}(uq^{-m})$ in $\tY(\fko_{m+n})$. Thus $\vartheta_m(\tY(\fko_n))$ is generated by the coefficients of $\tl s_{m+i,m+j}(u)$ for $1\lle i,j\lle n$. Now the statement follows immediately from Lemma \ref{stslem}.
\end{proof}

As an immediate consequence, we obtain some simple relations among the new generators. 

\begin{cor}\label{cor:commu-d-e-f}
We have $[d_i(u),d_j(u)]=0$ for all $i\ne j$, and
\begin{align*}
&[e_i(u),e_j(v)]=[e_i(u),f_j(v)]=[f_i(u),f_j(v)]=0, &\text{for }|i-j|>1,\\
&[d_i(u),e_{j}(v)]=[d_i(u),f_j(v)]=0,&\text{ for }i\ne j,j+1.
\end{align*}
\end{cor}
\begin{proof}
We prove for example if $i\ne j,j+1$, then $[d_i(u),e_j(v)]=0$. The other relations can be proved similarly. 

We first consider the case $i<j$. Let $m=i$, then the coefficients of $d_i(u)$ are contained in the subalgebra $\jmath(\tY(\fko_m))$ by \eqref{eq:sij-Gauss} while the coefficients of $e_{j}(v)=-\tl e_{j}(v)$ are contained in the subalgebra $\vartheta_m(\tY(\fko_{n-m}))$ by Lemma \ref{lem:theta} and \eqref{sij-Gauss-Inv}. It follows from Corollary \ref{comsubalg} that $[d_i(u),e_{j}(v)]=0$. For the case $i>j$, one proves similarly $[\tl d_i(u),e_{j}(v)]=0$ by setting $m=j+1$. Thus, we again have $[d_i(u),e_{j}(v)]=0$.
\end{proof}




\begin{lem}\label{lem:ei-generate-eij}
For $1\lle i<j<n$, we have
\beq\label{eq:ei-generate-eij}
(q^{-1}-q)e_{i,j+1}(u)=[e_{ij}(u),f_{j+1,j}^{(0)}]_q,\qquad (q^{-1}-q)f_{j+1,i}(u)=[f_{ji}(u),f_{j+1,j}^{(0)}]_q.
\eeq
In particular, the algebra $\tY(\fko_n)$ is generated by coefficients of $d_i(u)$, $e_j(u)$, and $f_j(u)$, for $1\lle i\lle n$ and $1\lle j<n$.
\end{lem}
\begin{proof}
We prove the first identity by induction on $i$. Note that $f_{j+1,j}^{(0)}=s_{j,j+1}^{(0)}$. By \eqref{eq:sij-Gauss} and \eqref{eq:zero-mode-s}, we have
\[
\big[d_1(u)e_{1j}(u),f_{j+1,j}^{(0)}\big]_q=(q^{-1}-q)d_1(u)e_{1,j+1}(u).
\]
The base case $i=1$ follows from this and the fact that the (invertible) $d_1(u)$ commutes with $f_{j+1,j}^{(0)}$ by Corollary \ref{cor:commu-d-e-f}.

Again by \eqref{eq:sij-Gauss} and \eqref{eq:zero-mode-s}, we have
\beq\label{expand ss}
\begin{split}
\Big[d_i(u)e_{ij}(u)&+\sum_{k<i}f_{ik}(u)d_k(u)e_{kj}(u),f_{j+1,j}^{(0)}\Big]_q\\
&=(q^{-1}-q)\Big(d_i(u)e_{i,j+1}(u)+\sum_{k<i}f_{ik}(u)d_k(u)e_{k,j+1}(u)\Big).
\end{split}
\eeq
It follows from Lemma \ref{comsubalg} and Corollary \ref{cor:commu-d-e-f} that, for $k<i<j$, we have
\[
[d_i(u),f_{j+1,j}^{(0)}]=[d_k(u),f_{j+1,j}^{(0)}]=[f_{ik}(u),f_{j+1,j}^{(0)}]=0.
\]
By these identities together with the induction hypothesis $$
[e_{kj}(u),f_{j+1,j}^{(0)}]_q=(q^{-1}-q)e_{k,j+1}(u),
$$ 
we simplify \eqref{expand ss} to become
$$
d_i(u)[e_{ij}(u),f_{j+1,j}^{(0)}]_q=(q^{-1}-q)d_i(u)e_{i,j+1}(u),
$$
proving the desired identity in \eqref{eq:ei-generate-eij}. 

The second identity in \eqref{eq:ei-generate-eij} follows from a similar calculation. 
\end{proof}

\subsection{Special twisted $q$-Yangians}
The twisted $q$-Yangian $\tY(\fko_n)$ is a $\gl_n$ analogy of affine $\imath$quantum group, cf. \cite{DF93}. We shall need the $\mathfrak{sl}_n$ version as a subalgebra of $\tY(\fko_n)$.

The \textit{special twisted $q$-Yangian} $\SY(\fko_n)$ is by definition the subalgebra of $\tY(\fko_n)$ generated by the coefficients of $\tl d_i(u)d_{i+1}(u)$, $e_{ij}(u)$, and $f_{ji}(u)$, for $1\lle i<j\lle n$.

For later use, we shall define the following generating series, for $1\lle i<j \lle n$:
\begin{align}
     \label{bdef}
\begin{split}
\bB_{ji}(u)&=\sum_{r\in\bZ} B_{ji,r}u^{r}=\frac{f_{ji}(u^{-1}q^{-i})-q^{-1}e_{ij}(uq^{-i})}{q-q^{-1}},\\
\bB_i(u)&=\sum_{r\in \bZ}B_{i,r}u^{r}:= \bB_{i+1,i}(u),
\end{split}
\\
\label{hdef}
\begin{split}
\bTH_0(u)&:=  1+ \sum_{r>0} (q-q^{-1})\acute\Theta_{0,r}u^{r}:=d_1(u^{-1}),\\
\bTH_i(u)&:= 1+ \sum_{r>0} (q-q^{-1}) \acute\Theta_{i,r}u^{r} =\tl d_{i}(u^{-1}q^{-i})d_{i+1}(u^{-1}q^{-i}),
\end{split}
\\
\BTH_i(u)&:=\frac{1-u^2}{1-q^{-2}u^2}\bTH_i(u),\qquad 0\lle i<n,\label{BTH}
\end{align}
cf. \cite[(3.41) \& Lemma 2.9]{LW21}. The elements (imaginary root vectors) $\acute\Theta_{i,r}$ correspond to the ones introduced in \cite{BK20} while the elements (imaginary root vectors) $\Theta_{i,r}$ are motivated from the $\imath$Hall algebra realization of the $q$-Onsager algebra \cite{LRW23}.

In terms of the new notations in \eqref{bdef}--\eqref{BTH}, Corollary~\ref{theta} can be rephrased as follows. 

\begin{cor} \label{theta2}
For the map $\vartheta_m\colon \tY(\fko_n) \longrightarrow \tY(\fko_{m+n})$, we have
\[
\vartheta_m(\bTH_i(u))=\bTH_{m+i}(u),\qquad \vartheta_m(\bB_i(u))=\bB_{m+i}(u).
\]
\end{cor}
\begin{lem}\label{lem:special-in-Gauss}
The special twisted $q$-Yangian $\SY(\fko_n)$ is generated by the coefficients of  $\bB_{i}(u)$ and $\bTH_i(u)$ \emph{(}or $\BTH_i(u)$\emph{)}, for $1\lle i<n$.
\end{lem}
\begin{proof}
This follows immediately from Lemma \ref{lem:ei-generate-eij} as $\bB_{i}(u)$ contains the coefficients of $e_{i}(u)$ and $f_i(u)$.
\end{proof}

Below we formulate the a weaker version of Proposition \ref{prop:PBW} on the PBW type bases for $\tY(\fko_n)$ and $\SY(\fko_n)$.

Let $V=\tY(\fko_n)$. Recall that we have the subalgebra $V_{\mc A}=V\cap \rU_{\mc A}$.
\begin{lem}\label{limitlem}
The algebra $V_{\mc A}$ is generated by 
\[
\dfrac{f_{ji}^{(r)}}{q-q^{-1}},\qquad \frac{e_{ij}^{(m)}}{q-q^{-1}},\qquad \frac{d_{k}^{(m)}}{q-q^{-1}},
\]
where $1\lle i<j\lle n$, $1\lle k\lle n$, $r\in\bZ_{\gge 0}$ and $m\in\bZ_{>0}$. Moreover, under the isomorphism \eqref{spec}, the images of these elements are given by
\[
\mathsf e_{ji}\otimes z^r-\mathsf e_{ij}\otimes z^{-r},\quad \mathsf e_{ij}\otimes z^m-\mathsf e_{ji}\otimes z^{-m},\quad \sfe_{kk}\otimes z^m-\sfe_{kk}\otimes z^{-m},
\]
respectively.
\end{lem}
\begin{proof}
The lemma follows from \eqref{eq:sij-Gauss} using \eqref{sigmaimg} via an obvious double induction on $i$ and then on $r$ (or $m$).
\end{proof}

\begin{cor}\label{corPBW}
Under the isomorphism \eqref{spec}, the images of $B_{ji,r}$, $\acute{\Theta}_{i,m}$, and $\acute{\Theta}_{0,m}$ are given by
\[
\mathsf e_{ji}\otimes z^r-\mathsf e_{ij}\otimes z^{-r},~~(\sfe_{i+1,i+1}-\sfe_{ii})\otimes z^m-(\sfe_{i+1,i+1}-\sfe_{ii})\otimes z^{-m},~~\sfe_{11}\otimes z^m-\sfe_{11}\otimes z^{-m},
\]
where $1\lle i<j\lle n$, $r\in\bZ$, and $m\in\bZ_{>0}$.
\end{cor}

\begin{prop}\label{PBWgauss}
Give any total order on the set of generators $B_{ji,r}$ and $\acute{\Theta}_{k,m}$, where $r\in\bZ$, $m\in\bZ_{>0}$, $1\lle i<j\lle n$ , $0\lle k<n$ \emph{(}resp. $1\lle k<n$\emph{)}, such that $\acute{\Theta}_{k,m}\prec B_{ji,r}$, then the ordered monomials in these generators are linearly independent in $\tY(\fko_n)$ \emph{(}resp. $\SY(\fko_n)$\emph{)}. 
\end{prop}
\begin{proof}
This follows immediately from Corollary \ref{corPBW} and the isomorphism \eqref{spec} as the images of the ordered monomials in these generators are linearly independent in $\rU(\gl_n[z^{\pm1}]^\theta)$.
\end{proof}

\section{Main results}\label{sec:main}

In this section we formulate our main results, including an explicit isomorphism between special twisted $q$-Yangians $\SY(\fko_n)$ and affine $\imath$quantum groups of type AI and a PBW theorem for affine $\imath$quantum groups as a corollary. 

\subsection{Main results}

Let $A=(c_{ij})_{1\lle i,j<n}$ be the Cartan matrix of type A$_{n-1}$. 

The \textit{affine $\imath$quantum group $\mathbf{U}^{\imath}$ of type AI in its Drinfeld presentation} \cite[Definition 3.10]{LW21} is the unital associative $\bC(q)$-algebra with generators $H_{i,r}$ and $B_{i,l}$, where $1\lle i<n$, $r\in\bZ_{>0}$ and $l\in\bZ$, subject to the following relations, for $m_1,m_2\in \bZ_{>0}$ and $k,l,k_1,k_2\in\bZ$:
\begin{align}
&[H_{i,m_1},H_{j,m_2}]=0,\label{relfirst}\\
&[H_{i,m},B_{j,l}]=\frac{[mc_{ij}]}{m}B_{j,l+m}-\frac{[mc_{ij}]}{m}B_{j,l-m},\label{hbcomp}\\
&[B_{i,k},B_{j,l}]=0,\hskip5.7cm \text{ if }c_{ij}=0,\\
&[B_{i,k},B_{j,l+1}]_{q^{-c_{ij}}}-q^{-c_{ij}}[B_{i,k+1},B_{j,l}]_{q^{c_{ij}}}=0,\qquad \text{ if }i\ne j,\\
&[B_{i,k},B_{i,l+1}]_{q^{-2}}-q^{-2}[B_{i,k+1},B_{i,l}]_{q^2}\notag\\
&\hskip 2cm =q^{-2}\Theta_{i,l-k+1}-q^{-4}\Theta_{i,l-k-1}+q^{-2}\Theta_{i,k-l+1}-q^{-4}\Theta_{i,k-l-1},\\
&B_{i,k_1}B_{i,k_2}B_{j,l}-[2]B_{i,k_1}B_{j,l}B_{i,k_2}+B_{j,l}B_{i,k_1}B_{i,k_2}+\{k_1\leftrightarrow k_2\}\notag\\
&=\Big(-\sum_{r\gge 0}q^{2r}[2][\Theta_{i,k_2-k_1-2r-1},B_{j,l-1}]_{q^{-2}}\notag\\
&\hskip 0.7cm -\sum_{r\gge 1}q^{2r-1}[2][B_{j,l},\Theta_{i,k_2-k_1-2r}]_{q^{-2}}-[B_{j,l},\Theta_{i,k_2-k_1}]_{q^{-2}}\Big)+\{k_1\leftrightarrow k_2\}.\label{rellast}
\end{align}
Here we set
\[
\Theta_{i,0}=\frac{1}{q-q^{-1}},\quad \Theta_{i,r}=0,\quad \text{ if }r<0,
\]
and $H_{i,r}$ are related to $\Theta_{i,r}$ by the following equation:
\beq\label{Th-H-rel}
1+\sum_{r\gge 0}(q-q^{-1})\Theta_{i,r}u^r=\exp\Big((q-q^{-1})\sum_{r\gge 1}H_{i,r}u^r\Big).
\eeq
Note that we specialize $\mathbb K_i$ and $C$ in \cite{LW21} to $1$; see \S\ref{sec:comments} for further discussions.

Similarly to \eqref{bdef}--\eqref{BTH}, introduce the generating series $\bB_i(u),\BTH_i(u),\bTH_i(u)$ and elements $\acute\Theta_{i,r}$ for affine $\imath$quantum group $\Ui$,
\begin{align*}
\bB_i(u)&:=\sum_{r\in\bZ}B_{i,r}u^r,\\
\BTH_i(u)&:=1+\sum_{r\gge 0}(q-q^{-1})\Theta_{i,r}u^r,\\
\bTH_i(u)&:=1+\sum_{r\gge 0}(q-q^{-1})\acute\Theta_{i,r}u^r:=\frac{1-q^{-2}u^2}{1-u^2}\BTH_i(u).
\end{align*}
Let $\delta(u)=\sum_{k\in\bZ}u^k$ be the formal delta function. 

\begin{prop}[{\cite[Theorems 5.1, 5.3]{LW21}}]\label{prop:series}
The above relations \eqref{relfirst}--\eqref{rellast} can be written in terms of generating series:
\begin{align}
&\bTH_i(u)\bTH_j(v) =\bTH_j(v)\bTH_i(u),\label{hhuo3}\\
&\bTH_i(u)\bB_j(v) =\frac{(1-q^{-c_{ij}}uv^{-1})(1-q^{c_{ij}}uv)}{(1-q^{c_{ij}}uv^{-1})(1-q^{-c_{ij}}uv)}\bB_j(v)\bTH_i(u),\label{hbuo3}\\
&\bB_i(u)\bB_j(v)=\bB_j(v)\bB_i(u),\qquad \qquad \hskip 3.3cm \text{ if }c_{ij}=0,\label{bbu1o3}\\
&(q^{c_{ij}}u-v)\bB_i(u)\bB_j(v)+(q^{c_{ij}}v-u)\bB_j(v)\bB_i(u)=0,\qquad  \text{ if }i\ne j,\label{bbu2o3}\\
&(q^2u-v) \bB_i(u) \bB_i(v) +(q^2v-u) \bB_i(v) \bB_i(u)\nonumber\\
&\hskip 3.4cm=\frac{1}{q-q^{-1}} \delta(uv) (u-v)\big(\bTH_i(v)-\bTH_i(u) \big),\label{bbu3o3}\\
&\bB_i(u_1)\bB_{i}(u_2)\bB_j(v)-[2]\bB_i(u_1)\bB_j(v)\bB_i(u_2)+\bB_j(v)\bB_i(u_1)\bB_i(u_2)+\{u_1\leftrightarrow u_2\}\notag\\
=&-\frac{\delta(u_1u_2)}{q-q^{-1}}\Big([\BTH_i(u_2),\bB_j(v)]_{q^{-2}}\frac{[2]vu_1^{-1}}{1-q^2u_2u_1^{-1}}+\label{serreuo3} \\&\hskip 3.3cm[\bB_j(v),\BTH_i(u_2)]_{q^{-2}}\frac{1+u_2u_1^{-1}}{1-q^2u_2u_1^{-1}}\Big)+\{u_1\leftrightarrow u_2\},\quad \text{if }c_{ij}=-1.\notag
\end{align}
\end{prop}
The relation \eqref{bbu3o3} has the following equivalent form,
\beq\label{bbunew}
\begin{split}
(q^2u-v) \bB(u) \bB(v)& +(q^2v-u) \bB(v) \bB(u)\\&=\frac{q^{-2}}{q-q^{-1}} \delta(uv) \big((q^2u-v)\BTH(v)+(q^2v-u)\BTH(u) \big);    
\end{split}
\eeq
see \cite[\S2.5]{LW21}.

\medskip 

Recall the elements $B_{i,r}$, $\Theta_{i,m}$, and $\acute\Theta_{i,m}$ of $\tY(\fko_n)$ obtained by Gauss decomposition of the matrix $S(u)$; see \eqref{bdef}--\eqref{BTH} in \S\ref{sec:Gauss}. 

Note that we have used the same notations for two algebras $\mathbf U^\imath$ and $\SY(\fko_n)$ defined in two different ways.
\begin{thm}
There is a $\bC(q)$-algebra isomorphism $\Phi:\mathbf U^\imath\to \SY(\fko_n)$, which sends
\[
B_{i,r}\mapsto B_{i,r},\quad \Theta_{i,m}\mapsto \Theta_{i,m},
\]
for $1\lle i<n$, $r\in\bZ$, and $m\in \bZ_{>0}$.
\end{thm}
\begin{proof}
We first prove the map sending $B_{i,k}\mapsto B_{i,k}$, $ \Theta_{i,m}\mapsto \Theta_{i,m}$ defines an surjective homomorphism.

It follows from Lemma \ref{lem:special-in-Gauss} that the special twisted $q$-Yangian $\SY(\fko_n)$ is generated by $B_{i,r}$ and $\Theta_{i,m}$ with $1\lle i<n$, $r\in\bZ$, and $m\in\bZ_{>0}$. Introduce new elements $H_{i,m}$ that are related to $\Theta_{i,m}$ (and hence $\acute\Theta_{i,m}$) by \eqref{Th-H-rel}. Then we prove that $B_{i,r}$, $\Theta_{i,m}$, $\acute\Theta_{i,m}$ and $H_{i,m}$ satisfy the relations \eqref{relfirst}--\eqref{rellast}. By Proposition \ref{prop:series}, these relations are equivalent to the relations \eqref{hhuo3}--\eqref{serreuo3} between generating series for $1\lle i,j<n$. We shall prove \eqref{hhuo3}--\eqref{serreuo3} by induction on $n$. The base cases $n=2,3$ will be established in Proposition \ref{sx2l} and Proposition \ref{thm:o3} below, respectively. Suppose now the relations hold for the special twisted $q$-Yangian $\SY(\fko_n)$, we would like to show it for the $\SY(\fko_{n+1})$ case. Recall that we have the natural homomorphism $\SY(\fko_{n})\to \SY(\fko_{n+1})$, $s_{ij}(u)\to s_{ij}(u)$ which sends $\bB_i(u)$, $\BTH_i(u)$ (and also $\bTH_i(u)$) of $\SY(\fko_{n})$ to the series in $\SY(\fko_{n+1})$ with the same name. On the other hand, we have the homomorphism $\vartheta_1:\SY(\fko_{n})\to \SY(\fko_{n+1})$ satisfying the properties in Corollaries \ref{theta} and \ref{theta2}. Therefore, the relations \eqref{hhuo3}--\eqref{serreuo3} hold for the case when $i,j<n$ or $i,j\gge 2$. It remains to prove the relations for the case when ($i=1$ and $j=n$) or ($i=n$ and $j=1$) which is obvious by Corollary \ref{cor:commu-d-e-f} if $n\gge 3$.

By the above observations, we have a surjective homomorphism
\beq\label{surj}
\Phi:\mathbf U^\imath\twoheadrightarrow \SY(\fko_n)
\eeq
which sends the generators $\Theta_{i,m}$ (or $H_{i,m}$), $B_{i,r}$ of $\mathbf U^\imath$ to the elements of $\SY(\fko_n)$ denoted by the same symbols. Thus it suffices to prove that the homomorphism $\Phi$ in \eqref{surj} is injective.

For any $1\lle i<j\lle n$, define elements $B_{ji,r}$ in $\mathbf U^\imath$ inductively by the rule:
\beq\label{bijr}
B_{i+1\, i,r}=B_{i,0},\qquad B_{ji,r}=[B_{j-1,0},B_{j-1\, i,r}]_q,
\eeq
if $j-i\gge 2$. Then it follows from Lemma \ref{lem:ei-generate-eij} that the homomorphism  $\Phi$ sends $B_{ji,r}$ in $\mathbf U^\imath$ to $B_{ji,r}$ in $\SY(\fko_n)$, where $B_{ji,r}$ in $\SY(\fko_n)$ are defined in \eqref{bdef}.

We shall prove that there exists a total order on the set of generators $B_{ji,r}$ and $\acute{\Theta}_{k,m}$, where $r\in\bZ$, $m\in\bZ_{>0}$, $1\lle i<j\lle n$, $0< k<n$, such that $\acute{\Theta}_{k,m}\prec B_{ji,r}$ and the ordered monomials in these elements span the affine $\imath$quantum group $\mathbf U^\imath$. 

It is clear by the relation \eqref{hbcomp} that the algebra $\mathbf U^\imath$ is spanned by the monomials 
\[
H_{j_1,m_1}H_{j_2,m_2}\cdots H_{j_l,m_l} B_{i_1,r_1}B_{i_2,r_2}\cdots B_{i_k,r_k}
\]
or equivalently (due to \eqref{relfirst} and \eqref{Th-H-rel}) by the monomials 
\beq\label{mono}
\acute\Theta_{j_1,m_1}\acute\Theta_{j_2,m_2}\cdots \acute\Theta_{j_l,m_l}B_{i_1,r_1}B_{i_2,r_2}\cdots B_{i_k,r_k}
\eeq
with various admissible indices. There is a filtration $\mathscr{F}_0\subset \mathscr{F}_1\subset \cdots\subset \mathscr F_k \subset \cdots \subset \mathbf U^\imath$ on $\mathbf U^\imath$, where $\mathscr F_k$ is the subspace spanned the monomials of the form \eqref{mono} (the number of generators $B_{i,r}$ occurs at most $k$ times).

Let $\gr\,\Ui$ be the associated graded algebra of $\Ui$ with respect to the above filtration. By abuse of notation, we shall denote the images of $B_{ji,r},\acute\Theta_{i,m}$ in $\gr\,\Ui$ again by $B_{ji,r},\acute\Theta_{i,m}$. Then it suffices to show that the ordered monomials in $B_{ji,r},\acute\Theta_{i,m}$ span $\gr\,\Ui$ which amounts to prove that any monomials $B_{i_1,r_1}B_{i_2,r_2}\cdots B_{i_k,r_k}$ is a linear combination of the ordered monomials in $B_{ji,r}$. 

Note that in $\gr\,\Ui$, the elements $B_{i,r}$ satisfy the relations
\begin{align*}
&[B_{i,k},B_{j,l}]=0,\hskip5.7cm \text{ if }c_{ij}=0,\\
&[B_{i,k},B_{j,l+1}]_{q^{-c_{ij}}}-q^{-c_{ij}}[B_{i,k+1},B_{j,l}]_{q^{c_{ij}}}=0,\\
&B_{i,k_1}B_{i,k_2}B_{j,l}-[2]B_{i,k_1}B_{j,l}B_{i,k_2}+B_{j,l}B_{i,k_1}B_{i,k_2}+\{k_1\leftrightarrow k_2\}=0,
\end{align*}
which are the defining relations of the negative half $\rU^{<}_q(\hgl_n)$; see \S\ref{subsec:loop}. Therefore, the subalgebra in $\gr\,\Ui$ generated by $B_{i,r}$ is a quotient of $\rU^{<}_q(\hgl_n)$. Moreover, it follows from \eqref{bijr-A} and \eqref{bijr} that the images of $\xi_{ji,r}$ correspond to $B_{ji,r}$ under this quotient. Now we pick the total order on $\xi_{ji,r}$ from \S\ref{subsec:loop}. This further induces a total order on $B_{ji,r}$ under the quotient. Then it follows from Proposition \ref{PBW-A} that any monomials $B_{i_1,r_1}B_{i_2,r_2}\cdots B_{i_k,r_k}$ is a linear combination of the ordered monomials in $B_{ji,r}$.  Finally, we extend the total order on $B_{ji,r}$ and $\acute\Theta_{k,m}$ such that $\acute\Theta_{k,m}\prec B_{ji,r}$. Note that the order between the elements $\acute\Theta_{k,m}$ is irrelevant as they commute. Therefore, we have proved that the ordered monomials in $B_{ji,r}$ and $\acute\Theta_{k,m}$ span the affine $\imath$quantum group $\mathbf U^\imath$. 

Now it follows from Proposition \ref{PBWgauss} that the surjective homomorphism $\Phi$ maps a spanning set of $\Ui$ to a linearly independent set of $\SY(\fko_n)$. Therefore, we obtain the injectivity of the homomorphism  $\Phi$, completing the proof of the main theorem.
\end{proof}

Recall that $A=(c_{ij})_{1\lle i,j<n}$ is the Cartan matrix of type A$_{n-1}$ and set $c_{01}=-1$.
\begin{thm}\label{mainthm2}
The twisted $q$-Yangian $\tY(\fko_n)$ is isomorphic to the unital associative $\bC(q)$-algebra with generators $H_{i,r}$ and $B_{j,l}$, where $0\lle i<n$, $1\lle j<n$, $r\in\bZ_{>0}$ and $l\in\bZ$, subject to the relations \eqref{relfirst}--\eqref{rellast}, for $m_1,m_2\in \bZ_{>0}$ and $k,l,k_1,k_2\in\bZ$.
\end{thm}
Note that $H_{0,m}$  $(r\in \bZ_{>0})$ are additional generators in $\tY(\fko_n)$, which are not present in $\SY(\fko_n)$. Also, we allow $i$ in $H_{i,m}$ to be 0 in the relations \eqref{relfirst}--\eqref{hbcomp}.

\begin{proof}
The proof is completely parallel to that of the main theorem.
\end{proof}

\subsection{Applications}
In this subsection, we discuss some obvious applications of our main results.

In the course of the proof of our main theorem, we also prove the following PBW theorem for affine $\imath$quantum group of type AI.

\begin{thm}\label{thm:PBW}
Give any total order on the set of generators $B_{ji,r}$ and $\acute{\Theta}_{k,m}$, where $r\in\bZ$, $m\in\bZ_{>0}$, $1\lle i<j\lle n$ , $0\lle k<n$ \emph{(}resp. $1\lle k<n$\emph{)}, such that \eqref{total} with $\xi$ replaced by $B$ is satisfied and $\acute{\Theta}_{k,m}\prec B_{ji,r}$, then the ordered monomials in these elements form a $\bC(q)$-basis of $\tY(\fko_n)$ \emph{(}resp. $\SY(\fko_n)$\emph{)}. Moreover, the statement remains true if all $\acute\Theta_{k,m}$ are replaced with  $H_{k,m}$ \emph{(}or $\Theta_{k,m}$\emph{)}. \qed
\end{thm}

Another application is that we have the following alternative descriptions of the relation between $\tY(\fko_n)$ and $\SY(\fko_n)$.

Define $\sdet\, S(u)=d_1(u)d_2(uq^{-2})\cdots d_n(uq^{-2n+2})$ and set
\beq\label{sdet}
\sdet\, S(u)=1+\sum_{r\gge 1} \mathfrak c_r u^{-r}.
\eeq

\begin{prop}\label{prop:center}
cf. \cite[Theorems 2.8.2, 2.9.2, 2.12.1]{Mol07}
\qquad
\begin{enumerate}
\item The coefficients $\mathfrak c_r$ with $r\in\bZ_{>0}$ of the series $\sdet\, S(u)$ are central in $\tY(\fko_n)$. 
\item Denote by $\mathfrak C_n^{\rm tw}$ the subalgebra of $\tY(\fko_n)$ generated by the elements $\mathfrak c_r$ with $r\in\bZ_{>0}$. Then the coefficients $\mathfrak c_r$ with $r\in\bZ_{>0}$ of $\sdet\, S(u)$ are algebraically independent generators of the subalgebra $\mathfrak C_n^{\rm tw}$ of $\tY(\fko_n)$. 
\item We have an algebra isomorphism,
$$
\tY(\fko_n)\cong \SY(\fko_n)\otimes \mathfrak C_n^{\rm tw}.
$$
\item The series $\sdet\, S(u)$ coincides with the Sklyanin determinant $\sdet\,S(u)$ defined in \cite[\S4]{MRS03}. 
\end{enumerate}
\end{prop}
\begin{proof}
The first statement follows from \eqref{hhuo3} and a straightforward calculation using \eqref{hbuo3}. The second and third statements are corollaries of Theorem \ref{thm:PBW} by noting \eqref{sdet} and $$\bTH_i(u^{-1}q^{-i})=\tl d_i(u)d_{i+1}(u),$$ where by convention $\tl d_0(u)=1$.

The proof of the last statement is similar to \cite[Theorem 2.12.1]{Mol07} with the application of the recurrence relation for Sklyanin minors established in the proof of \cite[Theorem 4.7]{MRS03}. In particular, the second statement also follows from the last one and \cite[Proposition 4.4]{MRS03}.
\end{proof}

Thus the special twisted $q$-Yangian $\SY(\fko_n)$ can also be regarded as a quotient of the twisted $q$-Yangian $\tY(\fko_n)$.
\begin{lem}\label{syquo}
The subalgebra $\SY(\fko_n)$ is isomorphic to the quotient of $\tY(\fko_n)$ by the ideal generated by the coefficients of $\sdet\,S(u)-1$, i.e., 
$\SY(\fko_n)\cong \tY(\fko_n)/(\sdet\,S(u)-1).$
\end{lem}

Let $\wp(u)$ and $\bar \wp(u)$ be any formal power series in $u^{-1}$ and $u$ with coefficients in $\bC$, respectively, 
\[
\wp(u)=\sum_{r\gge 0}\wp_ru^{-r}\in \bC[[u^{-1}]],\qquad \bar\wp(u)=\sum_{r\gge 0}\bar\wp_ru^{r}\in \bC[[u]],
\]
such that $\wp_0\bar\wp_0=1$. There is an automorphism of $\Uq(\hgl_n)$ defined by
\beq\label{eq:mu_f-A}
\mu_{\wp(u),\bar\wp(u)}:T(u)\mapsto \wp(u)T(u),\qquad \overline T(u)\mapsto \bar\wp(u)\overline T(u).
\eeq

The {\it quantum loop algebra} for $\mathfrak{sl}_n$ is the subalgebra $\Uq(\hsl_n)$ of $\Uq(\hgl_n)$ which consists of all elements stable under all the automorphisms of the form \eqref{eq:mu_f-A}; see \cite[Example 1.15.4]{Mol07}.





Let $\xi(u)$ be any formal power series in $u^{-1}$ with leading term $1$,
\[
\xi(u)=1+\sum_{r>0}\xi_ru^{-r}\in \bC[[u^{-1}]].
\]
There is an automorphism of $\tY(\fko_n)$ defined by
\beq\label{eq:mu_g-twisted}
\nu_{\xi(u)}:\tY(\fko_n) \longrightarrow \tY(\fko_n), \qquad S(u)\mapsto \xi(u)S(u).
\eeq

\begin{prop}[{cf. \cite[Theorem 11.7]{Kol14}}]
The special twisted $q$-Yangian is the subalgebra of $\tY(\fko_n)$ which consists of the elements stable under all automorphisms of the form \eqref{eq:mu_g-twisted}. Equivalently,
\[
\SY(\fko_n)=\Uq(\hsl_n)\cap \tY(\fko_n).
\]
\end{prop}
\begin{proof}
Let $\mathscr{SY}_n$ be the subalgebra of $\tY(\fko_n)$ which consists of the elements stable under all automorphisms of the form \eqref{eq:mu_g-twisted}. Then it is obvious that $\SY(\fko_n)\subset \mathscr{SY}_n$. By Proposition \ref{prop:center}, it suffices to show that 
\[
\tY(\fko_n)\cong \mathscr{SY}_n\otimes \mathfrak C_n^{\rm tw}.
\]
This is proved similarly as \cite[Theorem 1.8.2]{Mol07}.
\end{proof}

\subsection{Some comments}\label{sec:comments}
We conclude this section by adding some comments.

\noindent $\bullet$ Note that there are extra central elements $\mathbb K_i^{\pm 1}$ and $C^{\pm 1}$ in the universal affine $\imath$quantum group $\wtl{\mathbf U}^\imath$; see \cite[Definition 3.10 \& Theorem 5.1]{LW21}. It can easily seen that specializing these central elements to any nonzero scalars gives rise to an algebra isomorphic to $\Ui$. First, in the relations listed in \cite[Theorem 5.1]{LW21}, the element $C$ will be absent if we apply the substitutions,
\[
z\longmapsto zC^{-\frac12},\qquad w\longmapsto wC^{-\frac12}.
\]
The elements $\mathbb K_i^{\pm 1}$ can be eliminated from the relations as well by rescaling the series $\bB_i(u)$,
\[
\bB_i(u)\longmapsto \mathbb K_i^{-\frac12}\bB_i(u).
\]
Then the relations in \cite[Theorem 5.1]{LW21} reduces to the relations \eqref{hhuo3}--\eqref{serreuo3}.

\noindent $\Huge{\bullet}$ As pointed out in \cite[\S3.3]{MRS03}, one can construct coideal subalgebra from the quantum affine algebra $\Uq(\widehat{\gl}_n)$ with central charge $c$. However, the new matrix $S(u)$ still satisfies the same reflection equation \eqref{quat-mat} which does not involve the central charge $c$. Thus the corresponding subalgebra is isomorphic to the twisted $q$-Yangian, as an abstract algebra.

\section{Current relations in low ranks}\label{sec:low}

In this section, we establish the relations among the new current generators for twisted $q$-Yangians of rank one and two, following similar strategy of \cite{LWZ23a} adapted to the $q$-version.

\subsection{The rank 1 case}
In this subsection, we shall work with $\tY(\fko_2)$. Recall 
\beq \label{GD2}
\begin{split}
S(u)& =\begin{bmatrix} 1 & 0\\ f(u) & 1\end{bmatrix}\begin{bmatrix} d_1(u) & 0\\ 0 & d_2(u)\end{bmatrix}\begin{bmatrix} 1 & e(u)\\ 0 & 1\end{bmatrix}\\
&=\begin{bmatrix} d_1(u) & d_1(u)e(u)\\ f(u)d_1(u) & d_2(u)+f(u)d_1(u)e(u)\end{bmatrix},
\end{split}
\eeq
\beq\label{GD2inv}
\begin{split}
\wtl S(u)& =\begin{bmatrix} 1 & -e(u)\\ 0 & 1\end{bmatrix}\begin{bmatrix} \tl d_1(u) & 0\\ 0 & \tl d_2(u)\end{bmatrix}\begin{bmatrix} 1 & 0\\ -f(u) & 1\end{bmatrix}\\
&=\begin{bmatrix} \tl d_1(u)+e(u)\tl d_2(u)f(u) & -e(u)\tl d_2(u)\\ -\tl d_2(u)f(u) & \tl d_2(u)\end{bmatrix},
\end{split}
\eeq
where we dropped the subscripts for $e_1(u)$ and $f_1(u)$ in this rank 1 case. In this case, we have the R-matrices 
\beq\label{R2}
R(u,v)=\begin{bmatrix}
    uq^{-1} - q v& 0& 0& 0\\ 
    0& u - v&(q^{-1}- q) v& 0\\
    0&(q^{-1} - q) u& u - v&0\\
    0& 0& 0& uq^{-1} - q v
\end{bmatrix}
\eeq
and
\beq\label{R2t}
R^{t}(u^{-1},v)=\begin{bmatrix}
    u^{-1}q^{-1} - q v& 0& 0& (q^{-1} - q) u^{-1}\\ 
    0& u - v&0& 0\\
    0&0& u - v&0\\
    (q^{-1}- q) v& 0& 0& u^{-1}q^{-1} - q v
\end{bmatrix}
\eeq
in terms of Kronecker product. Similarly, the element $S_1(u)$ and $\wtl S_2(v)$ can be written as
\[
S_1(u)=\begin{bmatrix}
   s_{11}(u)& 0& s_{12}(u)& 0\\ 
    0& s_{11}(u)&0& s_{12}(u)\\
    s_{21}(u)&0& s_{22}(u)&0\\
    0& s_{21}(u)& 0& s_{22}(u)
\end{bmatrix}
\]
and 
\[
\wtl S_2(v)=\begin{bmatrix}
   \tl s_{11}(v)& \tl s_{12}(v)& 0& 0\\ 
    \tl s_{21}(v)& \tl s_{22}(v)&0& 0\\
    0&0& \tl s_{11}(v)&\tl s_{12}(v)\\
    0& 0& \tl s_{21}(v)& \tl s_{22}(v)
\end{bmatrix}
\]
We shall compare the entries of the matrices in the following equalities
\begin{align}
R(u,v)S_1(u)R^t(u^{-1},v)S_2(v)&=S_2(v) R^t(u^{-1},v)S_1(u)R(u,v),\label{ref2}\\
\wtl S_2(v) R(u,v)S_1(u)R^t(u^{-1},v)&=R^t(u^{-1},v)S_1(u)R(u,v)\wtl S_2(v).\label{ref-inv2}
\end{align}

\begin{lem}\label{lem:ddo2}
We have $[d_i(u),d_j(v)]=0$.
\end{lem}
\begin{proof}
By taking the $(1,1)$-entry (which corresponds to the coefficient of $E_{11}\otimes E_{11}$) of \eqref{ref2}, we have
\[
\Big(\frac{1}{qu}-qv\Big)\Big(\frac{u}{q}-qv\Big)[s_{11}(u),s_{11}(v)]=0.
\]
It follows from \eqref{GD2} that $d_1(u)=s_{11}(u)$. Therefore, $[d_1(u),d_1(v)]=0$. 

Applying $\vartheta_1$ to $[d_1(u),d_1(v)]=0$, we obtain $[d_2(u),d_2(v)]=0$ by Corollary \ref{theta}. The relation $[d_1(u), d_2(v)]=0$ follows from Corollary \ref{cor:commu-d-e-f}.
\end{proof}

\begin{lem}\label{d1efo2lem} 
We have
\beq\label{d1eo2}
\begin{split}
    (1-q^2)&(1-q^2uv)ud_1(u)e(u)-(1-q^2uv)(u-q^2v)d_1(u)e(v)\\
    &=-q^2(u-v)(1-uv)e(v)d_1(u)-q(1-q^2)(u-v)f(u)d_1(u).  
\end{split}
\eeq
and
\beq\label{d1fo2}
\begin{split}
    quv(1-q^2)&(u-v)d_1(u)e(u)+q^2(u-v)(1-uv)d_1(u)f(v)\\
    &=(u-q^2v)(1-q^2uv)f(v)d_1(u)-v(1-q^2)(1-q^2uv)f(u)d_1(u).
\end{split}
\eeq
\end{lem}
\begin{proof}
Taking the (1,2)-entry of \eqref{ref-inv2}, we have
\begin{align*}
\Big(\frac{1}{u}-v\Big)&(u-v)\tl s_{12}(v)s_{11}(u)=\Big(\frac{1}{qu}-qv\Big)\Big(\frac{u}{q}-qv\Big)s_{11}(u)\tl s_{12}(v)\\&+\Big(\frac{1}{q}-q\Big)\Big(1-\frac{v}{u}\Big)s_{21}(u)\tl s_{22}(v)+\Big(\frac{1}{q}-q\Big)\Big(\frac{1}{q}-quv\Big)s_{12}(u)\tl s_{22}(v).
\end{align*}
Transforming it in terms of new generating series by \eqref{GD2inv}, we obtain
\begin{align*}
-\Big(\frac{1}{u}-v\Big)&(u-v)e(v)\tl d_2(v)d_1(u)=-\Big(\frac{1}{qu}-qv\Big)\Big(\frac{u}{q}-qv\Big)d_1(u)e(v)\tl d_2(v)\\&+\Big(\frac{1}{q}-q\Big)\Big(1-\frac{v}{u}\Big)f(u)d_1(u)\tl d_2(v)+\Big(\frac{1}{q}-q\Big)\Big(\frac{1}{q}-quv\Big)d_1(u)e(u)\tl d_2(v).
\end{align*}
Noting that $[d_1(u),\tl d_2(v)]=0$, we move $\tl d_2(v)$ to the right. Canceling $\tl d_2(v)$ from right and clearing the denominator, the relation \eqref{d1eo2} is obtained.

The proof of \eqref{d1fo2} is similar by considering the (2,1)-entry.
\end{proof}

\begin{lem}\label{d2efo2lem} 
We have
\beq\label{d2eo2}
\begin{split}
(u-v)(1-q^4uv)&\tl d_2(v)e(u)-q(1-q^2)(u-v)\tl d_2(v)f(v)\\
&=(u-q^2v)(1-q^2uv)e(u)\tl d_2(v)-(1-q^2)(1-q^2uv)ve(v)\tl d_2(v)    
\end{split}
\eeq
and
\beq\label{d2fo2}
\begin{split}
u(1-q^2)(1-q^2uv)&\tl d_2(v)f(v)-(u-q^2v)(1-q^2uv)\tl d_2(v)f(u)\\
&=quv(1-q^2)(u-v)e(v)\tl d_2(v)-(u-v)(1-q^4uv)f(u)\tl d_2(v).   
\end{split}
\eeq
\end{lem}
\begin{proof}
Comparing the (2,4)-entries from \eqref{ref-inv2}, we have
\begin{align*}
(u-v)\Big(\frac{1}{qu}-qv\Big)\tl s_{22}(v)&s_{12}(u)+\Big(\frac{1}{q}-q\Big)\Big(\frac{1}{q}-\frac{qv}{u}\Big)\tl s_{21}(v)s_{11}(u)+ \Big(\frac{1}{q}-q\Big)^2\frac{v}{u}\tl s_{22}(v)s_{21}(u) \\
&=\Big(\frac{1}{q}-q\Big)\Big(\frac{1}{u}-v\Big)vs_{11}(u)\tl s_{12}(v)+\Big(\frac{u}{q}-qv\Big)\Big(\frac{1}{u}-v\Big)s_{12}(u)\tl s_{22}(v).
\end{align*}
In terms of the new generating series, it is equivalent to
\begin{align*}
(u-v)\Big(\frac{1}{qu}-qv\Big)&\tl d_2(v)d_1(u)e(u)-\Big(\frac{1}{q}-q\Big)\Big(\frac{1}{q}-\frac{qv}{u}\Big)\red{\tl d_{2}(v)f(v)d_1(u)}\\
&\hskip 3.7cm+ \Big(\frac{1}{q}-q\Big)^2\frac{v}{u}\red{\tl d_2(v)f(u)d_1(u)} \\
&=-\Big(\frac{1}{q}-q\Big)\Big(\frac{1}{u}-v\Big)vd_1(u)e(v)\tl d_{2}(v)+\Big(\frac{u}{q}-qv\Big)\Big(\frac{1}{u}-v\Big)d_1(u)e(u)\tl d_{2}(v).
\end{align*}
We further use \eqref{d1fo2} to move $d_1(u)$ in the 2-nd and 3-rd terms of the left-hand side from right to the left. Then we are able to cancel $d_1(u)$ from left. Simplifying the relation, one verifies \eqref{d2eo2}.

The relation \eqref{d2fo2} is proved similarly by using the (4,2)-entries and \eqref{d1eo2}.
\end{proof}

In order to see a symmetric form, we introduce the following generating series,
\beq\label{efhdef}
\mathcal E(u)=e(uq^{-1}),\qquad \mc F(u)=qf(u^{-1}q^{-1}),\quad \mc H(u)=\tl d_1(uq^{-1})d_2(uq^{-1}).
\eeq

We deduce the following corollary of Lemmas \ref{d1efo2lem} and \ref{d2efo2lem}.
\begin{lem}\label{dblemo2}
We have
\begin{align*}
d_1(u)\big(\mc E(v)-\mc F(v)\big)&=\frac{(1-q^{-1}uv)(1-quv^{-1})}{(1-quv)(1-q^{-1}uv^{-1})}\big(\mc E(v)-\mc F(v)\big)d_1(u),\\
\tl d_2(u)\big(\mc E(v)-\mc F(v)\big)&=\frac{(1-quv)(1-q^3uv^{-1})}{(1-q^3uv)(1-quv^{-1})}\big(\mc E(v)-\mc F(v)\big)\tl d_2(u).
\end{align*}
\end{lem}
\begin{proof}
 Substituting $v\to vq^{-1}$ in \eqref{d1eo2} and $v\to v^{-1}q^{-1}$ in \eqref{d1fo2}, we have
\begin{align*}
    (1-q^2)(1-quv)&ud_1(u)e(u)+(1-q^2)(qu-v)f(u)d_1(u)\\&=(1-quv)(u-qv)d_1(u)\mc E(v)-(qu-v)(q-uv)\mc E(v)d_1(u),
\end{align*}
\begin{align*}
    (1-q^2)(quv-1)&ud_1(u)e(u)+(1-q^2)(v-qu)f(u)d_1(u)\\&=(qu-v)(q-uv)\mc F(v)d_1(u)-(1-quv)(u-qv)d_1(u)\mc F(v).
\end{align*}
Adding two equations above, the l.h.s. is canceled and one easily verifies the first relation. The second relation follows from a similar calculation using Lemma \ref{d1efo2lem}.
\end{proof}

We shall also need the relations between $e(u)$ and $e(v)$, $e(u)$ and $f(v)$, $f(u)$ and $f(v)$.

\begin{lem}\label{lem:efo2}
We have
\begin{align}
\begin{split}
(q^2u&-v)\E(u)\E(v)+(q^2v-u)\E(v)\E(u)\label{eeo2}\\
&=\frac{q(1-q^2)(u-v)}{1-uv}\big(\mc H(u)-\mc H(v)\big)-(1-q^2)\big(u\E(u)\E(u)+v\E(v)\E(v)\big),
\end{split}\\
\begin{split}\label{efo2}
(q^2u&-v)\E(u)\F(v)+(q^2v-u)\F(v)\E(u)\\
&=\frac{q(1-q^2)(u-v)}{1-uv}\big(\mc H(u)-\mc H(v^{-1})\big)-(1-q^2)\big(u\E(u)\E(u)+v\F(v)\F(v)\big),
\end{split}\\
\begin{split}\label{ffo2}
(q^2u&-v)\F(u)\F(v)+(q^2v-u)\F(v)\F(u)\\
&=\frac{q(1-q^2)(u-v)}{1-uv}\big(\mc H(u^{-1})-\mc H(v^{-1})\big)-(1-q^2)\big(u\F(u)\F(u)+v\F(v)\F(v)\big).
\end{split}
\end{align}
\end{lem}
\begin{proof}
We first prove \eqref{eeo2}. It follows from \eqref{ref2} (the (1,4)-entry) that
\begin{align*}
 \Big(\frac{1}{u}-v\Big)s_{12}(u)s_{12}(v)&+ \Big(\frac{1}{q}-q\Big)\frac{1}{u}s_{11}(u)s_{22}(v)\\
 &=\Big(\frac{1}{u}-v\Big)s_{12}(v)s_{12}(u)+ \Big(\frac{1}{q}-q\Big)\frac{1}{u}s_{11}(v)s_{22}(u).
\end{align*}
Hence we have
\begin{align*}
 (1-uv)&d_1(u)e(u)d_1(v)e(v)+ \Big(\frac{1}{q}-q\Big)d_1(u)\big(d_2(v)+f(v)d_1(v)e(v)\big)\\
 &=(1-uv)d_1(v)e(v)d_1(u)e(u)+ \Big(\frac{1}{q}-q\Big)d_1(v)\big(d_2(u)+f(u)d_1(u)e(u)\big),
\end{align*}
which is equivalent to
\begin{align*}
 d_1(u)\Big((1-uv)&e(u)d_1(v)+\Big(\frac{1}{q}-q\Big)f(v)d_1(v)\Big)e(v)+ \Big(\frac{1}{q}-q\Big)d_1(u)d_2(v)\\
 &=d_1(v)\Big((1-uv)e(v)d_1(u)+\Big(\frac{1}{q}-q\Big)f(u)d_1(u)\Big)e(u)+ \Big(\frac{1}{q}-q\Big)d_1(v)d_2(u).
\end{align*}
Use \eqref{d1eo2} to move $d_1(v)$ (resp. $d_1(u)$) in
\[
(1-uv)e(u)d_1(v)+\Big(\frac{1}{q}-q\Big)f(v)d_1(v)
\]
(resp. the same formula with $u$ and $v$ interchanged) to the left. Multiplying $\tl d_1(u)\tl d_1(v)$ from left and using the substitution \eqref{efhdef}, the relation \eqref{eeo2} follows. 

The relation \eqref{efo2} can be proved using the (2,3)-entry of \eqref{ref-inv2} while the relation \eqref{ffo2} can be proved similarly using the (4,1)-entry of \eqref{ref2}.
\end{proof}

Next we work out the relations for $\tY(\so_2)$. Recall the generating series, \eqref{bdef}-\eqref{hdef}. Here we drop the index $i$ as the rank is 1.

\begin{prop}\label{sx2l}
We have the following relations in $\tY(\so_2)$ in generating series:
\begin{align}
&\bTH(u)\bTH(v) =\bTH(v)\bTH(u),\label{hhu}\\
&\bTH(u)\bB(v) =\frac{(1-q^{-2}uv^{-1})(1-q^2uv)}{(1-q^2uv^{-1})(1-q^{-2}uv)}\bB(v)\bTH(u),\label{hbu}\\
&(q^2u-v) \bB(u) \bB(v) +(q^2v-u) \bB(v) \bB(u)=\frac{1}{q-q^{-1}} \delta(uv) (u-v)\big(\bTH(v)-\bTH(u) \big).\label{bbu}
\end{align}
\end{prop}
\begin{proof}
The relations \eqref{hhu}, \eqref{hbu}, \eqref{bbu} follow from Lemmas \ref{lem:ddo2}, \ref{dblemo2}, \ref{lem:efo2}, respectively. 
\end{proof}

\subsection{The rank 2 case}

In the rank 2 case (i.e., $n=3$), we have the following expressions with ``$(u)$" omitted for the matrices $S(u)$ and $\tl S(u)$:
\begin{align} \label{SS3}
\begin{split}
    S&=\begin{bmatrix}
d_1 & d_1e_1& d_1e_{13}\\
f_1d_1 & d_2+f_1d_1e_1& d_2e_2+f_1d_1e_{13}\\
f_{31}d_1 & f_2d_2+f_{31}d_1e_{1} & d_3+f_2d_2e_2+f_{31}d_1e_{13}
\end{bmatrix},
\\
\wtl S&=\begin{bmatrix}
\tl d_1 +e_1\tl d_2 f_1+\tl e_{13}\tl d_3 \tl f_{31} & -e_1\tl d_2-\tl e_{13}\tl d_3 f_2& \tl e_{13}\tl d_3\\
-\tl d_2f_1-e_2\tl d_3 \tl f_{31} & \tl d_2+e_2\tl d_3f_1& -e_2\tl d_3\\
\tl d_3\tl f_{31} & -\tl d_3f_2 & \tl d_3
\end{bmatrix},
\end{split}
\end{align}
where $\tl e_{ij}$ and $\tl f_{ji}$ are defined in \eqref{eq:def-tilde-e-f}.

We introduce the following generating series,
\beq\label{efhdefo3}
\mc E_i(u)=e_i(uq^{-i}),\qquad \mc F_i(u)=qf_i(u^{-1}q^{-i}),\qquad \mc H_i(u)=\tl d_i(uq^{-i})d_{i+1}(uq^{-i}).
\eeq

\begin{lem}\label{o3eeff}
We have
\begin{align*}
(q^{-1}u-v)&\mc E_1(u)\mc E_2(v)+(q^{-1}v-u)\mc E_2(v)\mc E_1(u)\\
&=\big(q^{-1}-q\big)\big(q^{-1}v\mc E_1(vq^{-1})\mc E_2(v)-q^{-1}ve_{13}(vq^{-2})+ue_{13}(uq^{-1})\big),\\
(q^{-1}u-v)&\mc F_1(u)\mc F_2(v)+(q^{-1}v-u)\mc F_2(v)\mc F_1(u)\\
&=\big(q^{-1}-q\big)\big(v\mc F_2(v)\mc F_1(vq)-q^{2}vf_{31}(v^{-1}q^{-2})+quf_{31}(u^{-1}q^{-1})\big).
\end{align*}
\end{lem}
\begin{proof}
Taking the coefficients of $E_{12}\otimes E_{23}$ in \eqref{ref-inv2}, we find that
\begin{align*}
   (u-v)\tl s_{23}(v)s_{12}(u)&=\Big(\frac{1}{q}-q\Big)vs_{11}(u)\tl s_{13}(v)\\&+\Big(\frac{u}{q}-qv\Big)s_{12}(u)\tl s_{12}(v)+\Big(\frac{1}{q}-q\Big)us_{13}(u)\tl s_{33}(v).
\end{align*}
Rewrite it in terms of new generating series via \eqref{SS3}. Note that by Corollary \ref{cor:commu-d-e-f} we have
\[
[d_1(u),e_2(v)]=[\tl d_3(v),e_1(u)]=0.
\]
Thus we can move all $d_1(u)$ to the left and all $\tl d_3(v)$ to the right. Canceling $d_1(u)$ from left and $\tl d_3(v)$ from right, we obtain that
\begin{align*}
\Big(\frac{u}{q}-qv\Big)e_1(u)e_2(v)&+(v-u)e_2(v)e_1(u)\\&=\Big(\frac{1}{q}-q\Big)\big(ve_1(v)e_2(v)-ve_{13}(v)+ue_{13}(u)\big).
\end{align*}
Then the first relation follows from \eqref{efhdefo3}.

The proof of the second relation is similar by taking the coefficients of $E_{32}\otimes E_{21}$.
\end{proof}

\begin{lem}\label{o3eeff2}
We have
\begin{align*}
(q^{-1}u-v)&\mc E_1(u)\mc F_2(v)+(q^{-1}v-u)\mc F_2(v)\mc E_1(u)\\
&=\big(q^{-1}-q\big)\big(v\mc F_2(v)\mc F_1(vq)-q^{2}vf_{31}(v^{-1
}q^{-2})+ue_{13}(uq^{-1})\big),\\
(q^{-1}u-v)&\mc F_1(u)\mc E_2(v)+(q^{-1}v-u)\mc E_2(v)\mc F_1(u)\\
&=\big(q^{-1}-q\big)\big(q^{-1}v\mc E_1(vq^{-1})\mc E_2(v)-q^{-1}ve_{13}(vq^{-2})+quf_{31}(u^{-1}q^{-1})\big).
\end{align*}
\end{lem}
\begin{proof}
Set
\[
\Omega(u,v)=\Big(\frac{u}{q}-q v\Big)\tl s_{31}(v)s_{11}(u)+\Big(\frac{1}{q}-q \Big)v\tl s_{32}(v)s_{21}(u)+\Big(\frac{1}{q}-q\Big)v\tl s_{33}(v)s_{31}(u).
\]
Taking the coefficients of $E_{11}\otimes E_{31}$ and $E_{12}\otimes E_{32}$ in \eqref{ref-inv2}, we have
\begin{align*}
    \Big(\frac{1}{u} - v\Big) (u - v)s_{11}(u)\tl  s_{31}(v)&=\Big(\frac{1}{q} - q\Big)(u - v)v\tl s_{32}(v)s_{12}(u) \\&+ \Big(\frac{1}{q} - q\Big)(u-v)v\tl s_{33}(v)s_{13}(u)+\Big(\frac{1}{qu} - qv\Big)\Omega(u,v)
\end{align*}
and
\begin{align*}
    \Big(\frac{1}{u} - v\Big) (u - v)s_{12}(u)\tl  s_{32}(v)&=\Big(\frac{1}{qu} - qv\Big)(u - v)\tl s_{32}(v)s_{12}(u) \\&+ \Big(\frac{1}{q} - q\Big)(u-v)v\tl s_{33}(v)s_{13}(u)+\Big(\frac{1}{q} - q\Big)\frac{1}{u}\Omega(u,v).
\end{align*}
Eliminating $\Omega(u,v)$ from the two equations above, we obtain
\begin{align*}
    q(1-q^2uv)s_{12}(u)\tl s_{32}(v)&-(1-q^4uv)\tl s_{32}(v)s_{12}(u)\\
    &=q(1-q^2)s_{11}(u)\tl s_{31}(v)+q^2uv(1-q^2)\tl s_{33}(v)s_{13}(u).
\end{align*}
Proceeding as the proof of Lemma \ref{o3eeff}, we establish the first equality. 

The second equality is obtained by using the coefficients of $E_{11}\otimes E_{13}$ and $E_{21}\otimes E_{32}$ in \eqref{ref-inv2}.
\end{proof}

To establish the Serre relations, we need the Serre relations for degree zero generators.
\begin{lem}[{cf. \cite[Example 2.16.9]{Mol07}}]\label{Serre0:lem}
For $i\ne j$ \emph{(}i.e. $c_{ij}=-1$\emph{)}, we have
\[
B_{i,0}^2B_{j,0}-[2]B_{i,0}B_{j,0}B_{i,0}+B_{j,0}B_{i,0}^2=-q^{-1}B_{j,0}.
\]
\end{lem}
\begin{proof}
We only prove it for the case $i=1$ and $j=2$. The other case $i=2$ and $j=1$ is no different. By \eqref{bdef}, we have $f_{i}^{(0)}=(q-q^{-1})B_{i,0}$. Hence it suffices to work on $f_i^{(0)}$. By \eqref{eq:ei-generate-eij}, we have
\beq\label{o3pf1}
(q^{-1}-q)f_{31}^{(0)}=f_1^{(0)}f_{2}^{(0)}-qf_2^{(0)}f_1^{(0)}.
\eeq
Taking the coefficient of $u$ in the component of $E_{21}\otimes E_{31}$ in \eqref{ref2}, we find that
\[
-vs_{21}^{(0)}s_{31}(v)=-\frac{v}{q}s_{31}(v)s_{21}^{(0)}+(q^{-1}-q)\frac{v}{q}s_{32}(v),
\]
which further implies that
\beq\label{o3pf2}
f_{31}^{(0)}f_1^{(0)}-qf_1^{(0)}f_{31}^{(0)}=(q^{-1}-q)f_2^{(0)}.
\eeq
By \eqref{o3pf1} and \eqref{o3pf2}, one verifies that
\[
(f_1^{(0)})^2f_2^{(0)}-(q+q^{-1})f_1^{(0)}f_2^{(0)}f_1^{(0)}+f_2^{(0)}(f_1^{(0)})^2=-q^{-1}(q-q^{-1})^2f_2^{(0)}.
\]
Thus the desired identity follows.
\end{proof}
Next we work out the relations for $\tY(\fko_3)$.  Let $A=(c_{ij})_{i,j=1,2}$ be the Cartan matrix of type A$_2$. Recall the generating series, \eqref{bdef}-\eqref{hdef}. 
\begin{prop}\label{thm:o3}
The relations \eqref{hhuo3}--\eqref{serreuo3} hold in $\tY(\so_3)$.
\end{prop}
\begin{proof}
Similar to Lemma \ref{lem:ddo2}, one proves that $[d_i(u),d_j(v)]=0$ for all $i,j$. Thus \eqref{hhuo3} follows. The relation \eqref{hbuo3} essentially follows from Corollary \ref{cor:commu-d-e-f}, Lemma \ref{dblemo2}, and Corollary \ref{theta2} with suitable $\vartheta_m$. The relation \eqref{bbu1o3} is straightforward by Corollary \ref{cor:commu-d-e-f}. The relation \eqref{bbu2o3} is a direct corollary of Lemmas \ref{o3eeff}, \ref{o3eeff2} while the relation \eqref{bbu3o3} is obtained from \eqref{bbu} by applying the homomorphisms $\vartheta_0$ and $\vartheta_1$. Finally, as argued in \cite[\S4.6-\S4.8]{LW21}, if all other relations are valid, then the Serre relation \eqref{serreuo3} is reduced to the Serre relation for degree 0 generators, which is established in Lemma \ref{Serre0:lem}.
\end{proof}


\end{document}